\documentclass[10pt]{article}
\usepackage{amsmath,amssymb,amsthm,amsfonts,enumerate,times,
epsfig,color,wrapfig,graphicx,tikz,subfig,pdfpages,mathrsfs}
\usepackage[letterpaper,twoside,outer=1.3in,vmargin=1.3in,]{geometry}
\usepackage{kbordermatrix} 
\usepackage{hyperref,cleveref}
\hypersetup{
    colorlinks,
    citecolor=black,
    filecolor=black,
    linkcolor=black,
    urlcolor=black
}

\renewcommand{\arraystretch}{0.8}

\newtheorem{theorem}{Theorem}[section]

\newtheorem{LE}[theorem]{Lemma}
\newtheorem{PR}[theorem]{Proposition}
\newtheorem{CN}[theorem]{Conjecture}

\newtheorem{RE}[theorem]{Remark}

\newtheorem{DE}[theorem]{Definition}
\newcounter{claim_nb}[theorem]
\newtheorem{claim}[claim_nb]{Claim}
\newtheorem*{claim*}{Claim}

\newcommand{\tr}{\triangle}

\newcommand{\ignore}[1]{}
\newcommand{\conv}{\mathrm{conv}}

\newcommand{\supp}{\mathrm{support}}
\newcommand{\cuboid}{\mathrm{cuboid}}

\newcommand{\cycle}[1]{\mathrm{cycle}\!\left(#1\right)}
\newcommand{\cocycle}[1]{\mathrm{cocycle}\!\left(#1\right)}

\newcommand{\core}[1]{\mathrm{core}\!\left(#1\right)}
\newcommand{\setcore}[1]{\mathrm{setcore}\!\left(#1\right)}
\newcommand{\rank}[1]{\mathrm{rank}\!\left(#1\right)}
\newcommand{\depth}[1]{\mathrm{depth}\!\left(#1\right)}

\newenvironment{cproof}
{\begin{proof}
[Proof of Claim.]
\vspace{-1.2\parsep}}
{\renewcommand{\qed}{\hfill $\Diamond$} \end{proof}}

\title{Clean tangled clutters, simplices, and projective geometries}

\author{Ahmad Abdi \and G\'{e}rard Cornu\'{e}jols \and Matt Superdock}

\begin{document}
	
\maketitle
	
\begin{abstract} 
A clutter is \emph{clean} if it has no delta or the blocker of an extended odd hole minor, and it is \emph{tangled} if its covering number is two and every element appears in a minimum cover. Clean tangled clutters have been instrumental in progress towards several open problems on ideal clutters, including the $\tau=2$ Conjecture.

Let $\mathcal{C}$ be a clean tangled clutter. It was recently proved that $\mathcal{C}$ has a fractional packing of value two. Collecting the supports of all such fractional packings, we obtain what is called the {\it core} of $\mathcal{C}$. The core is a duplication of the cuboid of a set of $0-1$ points, called the {\it setcore} of $\mathcal{C}$.

In this paper, we prove three results about the setcore. First, the convex hull of the setcore is a full-dimensional polytope containing the center point of the hypercube in its interior. Secondly, this polytope is a simplex if, and only if, the setcore is the cocycle space of a projective geometry over the two-element field. Finally, if this polytope is a simplex of dimension more than three, then $\mathcal{C}$ has the clutter of the lines of the Fano plane as a minor.

Our results expose a fascinating interplay between the combinatorics and the geometry of clean tangled clutters.

\end{abstract}

\newpage

\tableofcontents

\newpage

\section{Introduction}

A \emph{clutter} is a family $\mathcal{C}$ of subsets of a finite set $V$ where no set contains another one~\cite{Edmonds70}. We refer to $V$ as the \emph{ground set}, to the elements in $V$ simply as \emph{elements}, and to the sets in $\mathcal{C}$ as \emph{members}. A \emph{transversal} is any subset of $V$ that intersects every member exactly once, whereas a \emph{cover} is any subset of $V$ that intersects every member at least once. A cover is \emph{minimal} if it does not contain another cover. The family of the minimal covers of $\mathcal{C}$ forms another clutter over ground set $V$; this clutter is called the \emph{blocker of $\mathcal{C}$} and is denoted $b(\mathcal{C})$. It is well-known that $b(b(\mathcal{C}))=\mathcal{C}$~\cite{Edmonds70,Isbell58}. Given disjoint $I,J\subseteq V$, the \emph{minor of $\mathcal{C}$} obtained after \emph{deleting $I$} and \emph{contracting $J$} is the clutter $\mathcal{C}\setminus I/J$ over ground set $V-(I\cup J)$ whose members are the inclusion-wise minimal sets in $\{C-J:C\in \mathcal{C}, C\cap I=\emptyset\}$. It is well-known that $b(\mathcal{C}\setminus I/J)=b(\mathcal{C})/I\setminus J$~\cite{Seymour76}.

A \emph{delta} is any clutter over a ground set of cardinality at least three, say $\{a_1,a_2,a_3,\ldots,a_n\}$, whose members are $\{a_2,a_3,\ldots,a_n\}$ and $\{a_1,a_i\}, i=2,\ldots,n$. (See \Cref{fig:delta-oddhole-K4}.) Observe that a delta is \emph{identically self-blocking}, that is, every member is a minimal cover, and vice versa. Observe further that the elements and members of a delta correspond to the points and lines of a degenerate projective plane.

An \emph{extended odd hole} is any clutter over a ground set of cardinality at least five and odd, say $\{a_1,\ldots,a_n\}$, whose minimum cardinality members are $\{a_1,a_2\},\{a_2,a_3\},\ldots,\{a_{n-1},a_n\},\{a_n,a_1\}$. That is, the minimum cardinality members correspond to the edges of an odd hole. (See \Cref{fig:delta-oddhole-K4}.) Note that there may exist other members, but those members would have cardinality at least three. Observe that every cover of an extended odd hole with $n$ elements has cardinality at least $\frac{n+1}{2}$.

\begin{DE}[\cite{Abdi-dyadic}] A clutter is \emph{clean} if it has no minor that is a delta or the blocker of an extended odd hole. \end{DE}

Observe that if a clutter is clean, then so is every minor of it. Clean clutters were introduced recently and a polynomial recognition algorithm was provided for them~\cite{Abdi-int-rest}. The class of clean clutters includes \emph{ideal} clutters, clutters without an \emph{intersecting minor}, and \emph{binary} clutters~\cite{Abdi-dyadic}.

The \emph{covering number} of a clutter $\mathcal{C}$, denoted $\tau(\mathcal{C})$, is the minimum cardinality of a cover.

\begin{DE}[\cite{Abdi-kwise}] A clutter is \emph{tangled} if it has covering number two and every element belongs to a minimum cover.\end{DE}

Observe that if a clutter has covering number at least two, then it has a tangled minor, obtained by repeatedly deleting elements that keep the covering number at least two. 

Let us define an important class of tangled clutters. A clutter is a \emph{cuboid} if its ground set can be relabeled $[2n]:=\{1,\ldots,2n\}$ for some integer $n\geq 1$, such that $\{1,2\},\{3,4\},\ldots,\{2n-1,2n\}$ are transversals~\cite{Abdi-mnp,Abdi-cuboids}. In particular, every member has cardinality $n$. Note that every cuboid without a cover of cardinality one is tangled. Consider, for instance, the clutter $Q_6=\{\{2,4,6\},\{1,3,6\},\{1,4,5\},\{2,3,5\}\}$, whose elements and members correspond to the edges and triangles of the complete graph $K_4$, as labeled in \Cref{fig:delta-oddhole-K4}. Then $Q_6$ is a cuboid -- as $\{1,2\},\{3,4\},\{5,6\}$ are transversals -- without a cover of cardinality one. Moreover, it can be readily checked that $Q_6$ has no minor that is a delta or the blocker of an extended odd hole. Consequently, $Q_6$ is a clean tangled clutter.

More generally, clean tangled clutters have been the subject of recent study as they have been instrumental in the progress made towards various outstanding problems on ideal clutters, ranging from recognizing idealness~\cite{Abdi-int-rest}, the $\tau=2$ Conjecture~\cite{Cornuejols00} and new examples of ideal minimally non-packing clutters~\cite{Abdi-infinite}, idealness of $k$-wise intersecting families~\cite{Abdi-kwise}, to the existence of dyadic fractional packings in ideal clutters~\cite{Abdi-dyadic}.

Even though their definition is purely combinatorial, clean tangled clutters enjoy fascinating geometric properties, and in this paper we initiate the study of the geometric attributes of such clutters. We prove three results that manifest an interplay between the geometry and the combinatorics of such clutters. In particular, full-dimensional simplices, projective geometries over the two-element field, and an astonishing connection between them play a central role in this work.

\begin{figure}
\centering
\includegraphics[scale=0.25]{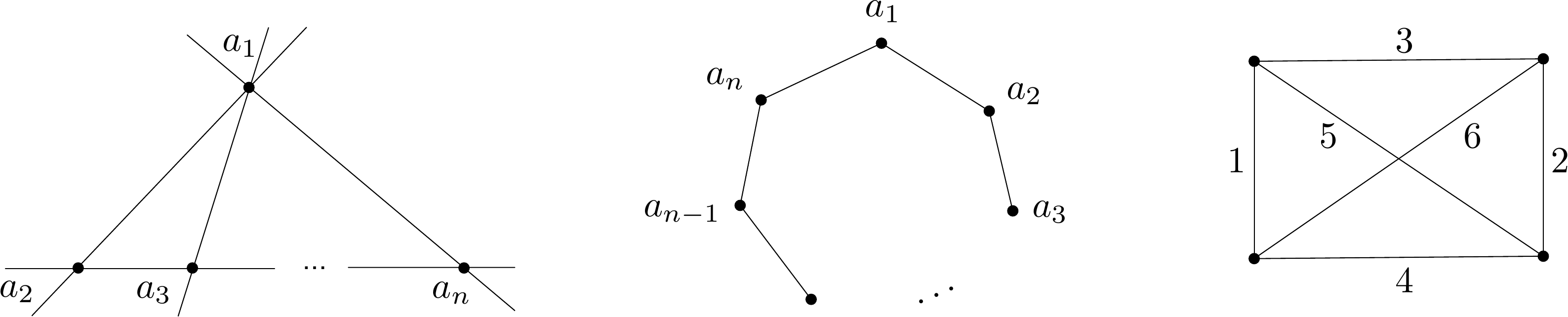}
\caption{{\bf Left:} The members of a \emph{delta} represented the lines of a degenerate projective plane. {\bf Middle:} The minimum cardinality members of an \emph{extended odd hole} represented as the edges of an odd hole. {\bf Right:} The members of $Q_6$ represented as the triangles of $K_4$.}
\label{fig:delta-oddhole-K4}
\end{figure}

\subsection{The core and the setcore of clean tangled clutters}

Let $\mathcal{C}$ be a clutter over ground set $V$. The \emph{incidence matrix of $\mathcal{C}$}, denoted $M(\mathcal{C})$, is the $0-1$ matrix whose columns are indexed by the elements and whose rows are indexed by the members. Consider the primal-dual pair of linear programs $$
(P)\quad\begin{array}{ll} \min \quad  &{\bf 1}^\top x\\ \text{s.t.} & M(\mathcal{C})x\geq {\bf 1}\\ & x\geq {\bf 0} \end{array}
\quad\quad\quad
(D)\quad
\begin{array}{ll} \max \quad  &{\bf 1}^\top y\\ \text{s.t.} & M(\mathcal{C})^\top y\leq {\bf 1}\\ & y\geq {\bf 0}.\end{array}
$$ The incidence vector of any cover of $\mathcal{C}$ gives a feasible solution for (P). Thus $\tau(\mathcal{C})$ is an upper bound on the optimal value of (P). A {\it fractional packing of $\mathcal{C}$} is any feasible solution $y$ for (D), and its {\it value} is ${\bf 1}^\top y$. Its {\it support}, denoted $\supp(y)$, is the clutter over ground set $V$ whose members are $\{C\in \mathcal{C}: y_C>0\}$. It follows from Weak LP Duality that every fractional packing has value at most $\tau(\mathcal{C})$. In general, this upper bound is far from being tight. However, what is fascinating about clean clutters is that,

\begin{theorem}[\cite{Abdi-hole}, Theorem 3 and \cite{Abdi-int-rest}, Lemma 1.6]\label{dense}
Every clean clutter with covering number at least two has a fractional packing of value two. In particular, every clean tangled clutter has a fractional packing of value two.
\end{theorem}

We may therefore make the following definition:

\begin{DE}
Let $\mathcal{C}$ be a clean tangled clutter. Then the $\emph{core}$ of $\mathcal{C}$ is the clutter $$\core{\mathcal{C}} = \{C \in \mathcal{C} : y_{C} > 0 \text{ for some fractional packing $y$ of value two}\}.$$
\end{DE}

By \Cref{dense}, every clean tangled clutter has a nonempty core. Let us identify the core for two examples of clean tangled clutters. For the first example, consider the clean tangled clutter $Q_6$. As $\left(\frac12,\frac12,\frac12,\frac12\right)\in \mathbb{R}^{Q_6}_+$ is a fractional packing of value two, it follows that $\core{Q_6}=Q_6$. For the second example, consider the clutter $Q$ whose incidence matrix is 
\renewcommand{\kbldelim}{(}
\renewcommand{\kbrdelim}{)}
\[
  M(Q) = \kbordermatrix{
        &1&2&3&4&5&6&7&8\\
	\frac12&1&1&0&0&1&0&1&0\\
	\frac12&1&1&0&0&0&1&0&1\\
	\frac12&0&0&1&1&1&0&0&1\\
	\frac12&0&0&1&1&0&1&1&0\\
	0&1&0&1&1&1&0&1&0\\
	0&0&1&1&0&0&1&0&1\\
	0&0&1&1&0&1&0&0&1\\
	0&1&1&0&1&0&1&1&0
  }.
\] $Q$ is an \emph{ideal minimally non-packing} clutter with covering number two~\cite{Cornuejols00}, implying in turn that it is a clean tangled clutter~\cite{Abdi-infinite}. The row labels indicate the unique fractional packing of value two, where uniqueness is a simple consequence of Complementary Slackness. Subsequently, $\core{Q}$ consists of the four members of $Q$ corresponding to the top four rows of $M(Q)$. 


Let $\mathcal{C}$ be a clutter over ground set $V$. Distinct elements $u,v$ are \emph{duplicates in $\mathcal{C}$} if the corresponding columns in $M(\mathcal{C})$ are identical. To {\it duplicate an element $w$ of $\mathcal{C}$} is to introduce a new element $\bar{w}$ and replace $\mathcal{C}$ by the clutter over ground set $V\cup \{\bar{w}\}$ whose members are $\{C:w\notin C\in \mathcal{C}\}\cup \{C\cup \{\bar{w}\}:w\in C\in \mathcal{C}\}$. A {\it duplication of $\mathcal{C}$} is any clutter obtained from $\mathcal{C}$ after duplicating some elements.

Looking back at $\core{Q}$, we see that elements $1,2$ are duplicates and elements $3,4$ are duplicates, and that $\core{Q}$ is a duplication of $Q_6$. In fact, the core of any clean tangled clutter is a duplication of a cuboid -- let us elaborate.

Take an integer $n\geq 1$ and a set $S\subseteq \{0,1\}^n$. The \emph{cuboid of $S$}, denoted $\cuboid(S)$, is the clutter over ground set $[2n]$ whose members have incidence vectors $\{(p_1,1-p_1,\ldots,p_n,1-p_n):p\in S\}$. In particular, there is a bijection between the members of $\cuboid(S)$ and the points in $S$. Observe that every cuboid is obtained in this way. For example, $Q_6$ is the cuboid of the set $\{000,110,101,011\}\subseteq \{0,1\}^3$, represented in \Cref{fig:R11}.

\begin{figure}
\centering
\includegraphics[scale=0.4]{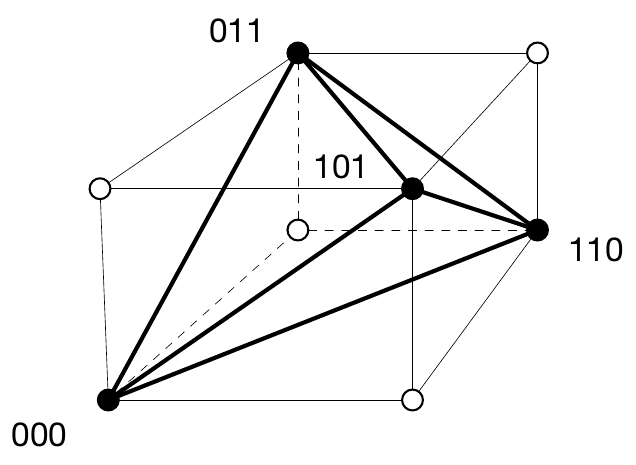}
\caption{A representation of $\setcore{Q_6}$ and its convex hull.}
\label{fig:R11}
\end{figure}

Take a point $q\in \{0,1\}^n$. To {\it twist $S$ by $q$} is to replace $S$ by $S\tr q:=\{p\tr q:p\in S\}$, where the second $\tr$ denotes coordinate-wise addition modulo $2$. Take a coordinate $i\in [n]$. Denote by $e_i$ the $i\textsuperscript{th}$ unit vector of appropriate dimension. To {\it twist coordinate $i$ of $S$} is to replace $S$ by $S\tr e_i$. Two sets $S_1,S_2$ are {\it isomorphic}, written as $S_1\cong S_2$, if one is obtained from the other after relabeling and twisting some coordinates. Two distinct coordinates $i,j\in [n]$ are {\it duplicates in $S$} if $S\subseteq \{x : x_i=x_j\}$ or $S\subseteq \{x : x_i + x_j=1\}$. Observe that if two coordinates are duplicates in a set, then they are duplicates in any isomorphic set. Observe further that $S$ has duplicated coordinates if, and only if, $\cuboid(S)$ has duplicated elements.

Let $\mathcal{C}$ be a clean tangled clutter over ground set $V$. Denote by $G(\mathcal{C})$ the graph over vertex set $V$ whose edges correspond to the minimum covers of $\mathcal{C}$. It can be readily checked that $G(\mathcal{C})$ is bipartite and every vertex of it is incident with an edge~\cite{Abdi-hole}. The {\it rank of $\mathcal{C}$}, denoted $\rank{\mathcal{C}}$, is the number of connected components of the bipartite graph $G(\mathcal{C})$. Our first result, below, justifies this choice of terminology. 

For example, the reader can verify that $G(Q_6),G(Q)$ are the bipartite graphs represented in \Cref{fig:Q6Q}, each of which has exactly three connected components, so $\rank{Q_6}=\rank{Q}=3$.

We are ready to state our first result:

\begin{theorem}[proved in \S\ref{sec:core}]\label{setcore}
Let $\mathcal{C}$ be a clean tangled clutter of rank $r$. Then there exists a set $S\subseteq \{0,1\}^r$ such that the following statements hold: 
\begin{enumerate}[(i)]
\item $\core{\mathcal{C}}$ is a duplication of $\cuboid(S)$, and up to isomorphism, $S$ is the unique set satisfying this property.
\item There is a one-to-one correspondence between the fractional packings of value two in $\mathcal{C}$ and the different ways to express $\frac12 \cdot{\bf 1}$ as a convex combination of the points in $S$.
\item $\conv(S)$ is a full-dimensional polytope containing $\frac12 \cdot{\bf 1}$ in its interior. 
\end{enumerate} 
\end{theorem}

\begin{DE}\label{setcore-defn} Let $\mathcal{C}$ be a clean tangled clutter of rank $r$. The \emph{setcore of $\mathcal{C}$}, denoted $\setcore{\mathcal{C}}$, is the unique set $S\subseteq \{0,1\}^r$ such that $\core{\mathcal{C}}$ is a duplication of $\cuboid(S)$.
\end{DE}

An explicit description of the setcore is provided in \S\ref{sec:core}. 

For example, we see that $\setcore{Q_6} = \setcore{Q} = \{000,110,101,011\}$. Notice further that the convex hull of $\{000,110,101,011\}$ is a full-dimensional polytope containing the point $(\frac12,\frac12,\frac12)$ in its interior. (See \Cref{fig:R11}.)

\begin{figure}
\centering
\includegraphics[scale=0.5]{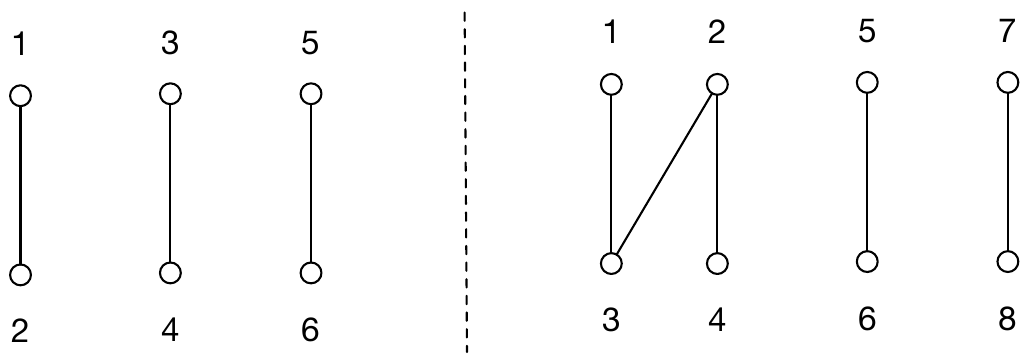}
\caption{{\bf Left:} A representation of $G(Q_6)$. {\bf Right:} A representation of $G(Q)$.}
\label{fig:Q6Q}
\end{figure}

\subsection{Simplices and projective geometries over the two-element field}

The convex hull of the setcore of any clean tangled clutter is a full-dimensional polytope by \Cref{setcore}~(iii). A natural geometric question arises: When is this polytope a simplex? Surprisingly, the answer to this basic question takes us to binary matroids. (Our terminology follows Oxley~\cite{Oxley11}.)

Let $k\geq 1$ be an integer, and let $A$ be the $k\times (2^k-1)$ matrix whose columns are $\left\{a\in \{0,1\}^k:a\neq {\bf 0}\right\}$. The binary matroid represented by $A$ is called the \emph{rank-$k$ projective geometry over $GF(2)$} and denoted $PG(k-1,2)$.\footnote{In the context of binary matroids, rank refers to $GF(2)$-rank, whereas in the context of clutters, rank refers to $\mathbb{R}$-rank.} Let $n:=2^k-1$. The \emph{cycle space of $PG(k-1,2)$} is $$\cycle{PG(k-1,2)}:=\big\{x\in \{0,1\}^n: Ax \equiv {\bf 0} \pmod{2}\big\}.$$ Observe that the cycle space forms a vector space over $GF(2)$. The \emph{cocycle space of $PG(k-1,2)$} is $$\cocycle{PG(k-1,2)} := \big\{ A^\top y ~~\text{mod $2$} : y\in \{0,1\}^k\big\}\subseteq \{0,1\}^n.$$ Observe that the cocycle space is the orthogonal complement of the cycle space. As the rows of $A$ are linearly independent over $GF(2)$, the cocycle space has $2^k=n+1$ points. In fact, we see in~\S\ref{sec:PG->simplex} that the $n+1$ points form the vertices of an $n$-dimensional simplex. Our second result, below, serves as a converse to this statement.

We refer to $PG(k-1,2),k\geq 1$ simply as \emph{projective geometries}. Let us look at the first three projective geometries. For the first one, $\cocycle{PG(0,1)} = \{0,1\}$. As for the second one, notice that $PG(1,2)$ is nothing but the graphic matroid of a triangle, and that $\cocycle{PG(1,2)}=\{000,110,101,011\}$. The third one, $PG(2,2)$, is known as the \emph{Fano matroid}. See \Cref{fig:PG-reps} for representations of these three matroids.

We are now ready to state our second result:

\begin{theorem}[$(\Leftarrow)$ proved in \S\ref{sec:PG->simplex}, $(\Rightarrow)$ proved in \S\ref{sec:simplex->PG}]\label{main-PG}
Let $\mathcal{C}$ be a clean tangled clutter. Then $\conv(\setcore{\mathcal{C}})$ is a simplex if, and only if, $\setcore{\mathcal{C}}$ is the cocycle space of a projective geometry.
\end{theorem}

For instance, as can be seen in \Cref{fig:R11}, the convex hull of $\setcore{Q_6}$ is a simplex, so according to \Cref{main-PG}, $\setcore{Q_6}$ is the cocycle space of a projective geometry. This is indeed the case as $\setcore{Q_6} =\{000,110,101,011\}= \cocycle{PG(1,2)}$.

\begin{figure}
\centering
$$
\begin{pmatrix}
1
\end{pmatrix}\qquad
\begin{pmatrix}
1&0&1\\
0&1&1
\end{pmatrix}\qquad
\begin{pmatrix}
1&0&0&0&1&1&1\\
0&1&0&1&0&1&1\\
0&0&1&1&1&0&1
\end{pmatrix}
$$
\caption{Representations of $PG(0,2),PG(1,2)$ and $PG(2,2)$, from left to right.}
\label{fig:PG-reps}
\end{figure}

\subsection{The clutter of the lines of the Fano plane}

Consider the clutter over ground set $\{1,\ldots,7\}$ whose members are $$\mathbb{L}_7 :=
\big\{
\{1,2,3\}, \{1,4,5\}, \{1,6,7\}, \{2,5,6\}, \{2,4,7\}, \{3,4,6\},\{3,5,7\}
\big\}.
$$ Note that the members of $\mathbb{L}_7$ correspond to the lines of the \emph{Fano plane}, as displayed in \Cref{fig:Fano}. The members of $\mathbb{L}_7$ may also be viewed as the lines (i.e. triangles) of the Fano matroid. Observe further that $\mathbb{L}_7$ is an identically self-blocking clutter. 

As the only \emph{minimally non-ideal binary clutter} with a member of cardinality three~\cite{Abdi-triangle}, $\mathbb{L}_7$ plays a crucial role in Seymour's \emph{Flowing Conjecture}, predicting an excluded minor characterization of \emph{ideal binary clutters}~\cite{Seymour77}. 

Our third result relates to finding $\mathbb{L}_7$ as a minor in clean tangled clutters:

\begin{theorem}[proved in \S\ref{sec:fano}]\label{main-L7}
Let $\mathcal{C}$ be a clean tangled clutter where $\conv(\setcore{\mathcal{C}})$ is a simplex. If $\rank{\mathcal{C}}>3$, then $\mathcal{C}$ has an $\mathbb{L}_7$ minor.
\end{theorem}

Let us outline a naive approach for proving this theorem. Though unsuccessful, this attempt explains the intuition behind \Cref{main-L7}. Let $\mathcal{C}$ be a clean tangled clutter where $\conv(\setcore{\mathcal{C}})$ is a simplex, and $\rank{\mathcal{C}}>3$. By \Cref{main-PG}, $\setcore{\mathcal{C}} = \cocycle{PG(k-1,2)}$ for some $k\geq 1$. As $2^k-1=\rank{\mathcal{C}}>3$ and $\rank{\mathcal{C}}=2^k-1$, we must have that $k\geq 3$. From here, the reader can verify that since $PG(k-1,2)$ has the Fano matroid as a minor, $\cuboid(\cocycle{PG(k-1,2)})$ has an $\mathbb{L}_7$ minor. Consequently, $\core{\mathcal{C}}$, which is a duplication of $\cuboid(\setcore{\mathcal{C}})$ by \Cref{setcore}, must have an $\mathbb{L}_7$ minor. However, this does \emph{not} necessarily imply that $\mathcal{C}$ has an $\mathbb{L}_7$ minor, because $\core{\mathcal{C}}$ is only a subset of $\mathcal{C}$, so minors of $\core{\mathcal{C}}$ do not necessarily correspond to minors of $\mathcal{C}$. In \S\ref{sec:fano}, we see an elaborate, successful attempt for proving \Cref{main-L7}. 

\begin{figure}
\centering
\includegraphics[scale=0.4]{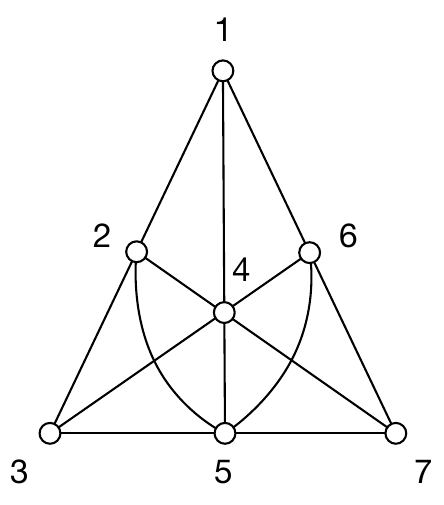}
\caption{The Fano plane}
\label{fig:Fano}
\end{figure}

\subsection{Outline of the paper}

In \S\ref{sec:core}, we prove \Cref{setcore} and provide applications of the theorem used in later sections.
In \S\ref{sec:PG->simplex}, after a primer on binary matroids, we show how every projective geometry leads to a simplex, and prove \Cref{main-PG} $(\Leftarrow)$ as a consequence.
In \S\ref{sec:simplex->PG}, we prove \Cref{main-PG} $(\Rightarrow)$, and derive an appealing consequence on characterizing simplices that come from projective geometries.
In \S\ref{sec:fano}, after laying some ground work, we prove \Cref{main-L7}, and then discuss an application to idealness.
Finally, in \S\ref{sec:conclusion}, we discuss future directions for research, and conclude with two conjectures.

\section{The core and the setcore of clean tangled clutters}\label{sec:core}

In this section, after presenting some lemmas, we prove \Cref{setcore}, and then provide three applications for clean tangled clutters: the first is a characterization of the core, the second is an explicit description of the setcore when the rank is small, and the third is an equivalent condition for having a simplicial setcore.

Given a clean tangled clutter $\mathcal{C}$ over ground set $V$, recall that $G(\mathcal{C})$ denotes the graph over vertex set $V$ whose edges correspond to the minimum covers of $\mathcal{C}$. We need the following theorem:

\begin{theorem}[\cite{Abdi-dyadic}]\label{bipartite}
Let $\mathcal{C}$ be a clean tangled clutter. Then $G(\mathcal{C})$ is a bipartite graph where every vertex is incident with an edge. Moreover, if $G(\mathcal{C})$ is not connected and $\{U,U'\}$ is the bipartition of a connected component, then $\mathcal{C}\setminus U/U'$ is a clean tangled clutter.
\end{theorem}

\subsection{Some lemmas}\label{sec:core/lemmas}

Let $\mathcal{C}$ be a clutter over ground set $V$. Consider the primal-dual pair of linear programs $$
(P)\quad\begin{array}{ll} \min \quad  &{\bf 1}^\top x\\ \text{s.t.} & \sum_{v\in C} x_v\geq 1 \quad C\in \mathcal{C}\\ & x\geq {\bf 0} \end{array}
\quad\quad\quad
(D)\quad
\begin{array}{ll} \max \quad  &{\bf 1}^\top y\\ \text{s.t.} & \sum\left(y_C:v\in C\in \mathcal{C}\right)\leq 1 \quad v\in V\\ & y\geq {\bf 0}.\end{array}
$$ By applying Complementary Slackness to this pair, we get the following:

\begin{RE}\label{CS-conditions} 
Let $\mathcal{C}$ be a clutter, $B$ a minimum cover, and $y$ a fractional packing of value $\tau(\mathcal{C})$, if there is any. Then $|C\cap B| = 1$ for every $C\in \mathcal{C}$ such that $y_C>0$, and $\sum\left(y_C:v\in C\in \mathcal{C}\right) = 1$ for every element $v\in B$.\end{RE}

\paragraph{An explicit description of the setcore.} Let $\mathcal{C}$ be a clean tangled clutter. Recall that $$\core{\mathcal{C}}=\{C\in \mathcal{C}:y_C>0 \text{ for some fractional packing $y$ of value two}\}.$$ The following is an immediate consequence of \Cref{CS-conditions}:

\begin{RE}\label{CS-conditions-2} 
Let $\mathcal{C}$ be a clean tangled clutter over ground set $V$. Then every member of $\core{\mathcal{C}}$ is a transversal of the minimum covers of $\mathcal{C}$. Moreover, for every fractional packing $y$ of value two and for every element $v\in V$, $\sum\left(y_C:v\in C\in \mathcal{C}\right) = 1$.
\end{RE}

Let $G:=G(\mathcal{C})$ and $r:=\rank{\mathcal{C}}$. By \Cref{bipartite}, $G$ is a bipartite graph where every vertex is incident with an edge, and it has $r$ connected components by definition. For each $i\in [r]$, denote by $\{U_i,V_i\}$ the bipartition of the $i\textsuperscript{th}$ connected component of $G$. As an immediate consequence of \Cref{CS-conditions-2},

\begin{RE}\label{core->cuboid}
Let $\mathcal{C}$ be a clean tangled clutter of rank $r$, and for each $i\in [r]$, denote by $\{U_i,V_i\}$ the bipartition of the $i\textsuperscript{th}$ connected component of $G(\mathcal{C})$. Let $C$ be a member of $\mathcal{C}$. If $C\in \core{\mathcal{C}}$, then $C\cap (U_i\cup V_i)\in \{U_i,V_i\}$ for each $i\in [r]$. \footnote{By the end of this section, we shall see that the converse of this remark also holds.}
\end{RE}

In particular, $\core{\mathcal{C}}$ is a duplication of a cuboid -- let us elaborate. Consider the set $S\subseteq \{0,1\}^r$ defined as follows: start with $S=\emptyset$, and for each $C\in \core{\mathcal{C}}$, add a point $p$ to $S$ such that $$p_i = 0 \quad\Leftrightarrow\quad C\cap (U_i\cup V_i) = U_i \quad \forall i\in [r].$$ By \Cref{core->cuboid}, the set $S$ is well-defined and $\core{\mathcal{C}}$ is a duplication of $\cuboid(S)$. Thus this must be the unique set foreseen by \Cref{setcore}~(i), and called the setcore of $\mathcal{C}$ by \Cref{setcore-defn}. We call $S$ the {\it setcore of $\mathcal{C}$ with respect to $(U_1,V_1;U_2,V_2;\ldots;U_r,V_r)$}, and denote it $\setcore{\mathcal{C}:U_1,V_1;U_2,V_2;\ldots;U_r,V_r}$. (Note that we have not yet proved uniqueness, though we will see a proof shortly.) 

\paragraph{Fractional packings vs. convex combinations.} The following remark, which is an immediate consequence of \Cref{CS-conditions}, sheds light on how the hypercube center point $\frac12\cdot {\bf 1}$ comes into play in \Cref{setcore}:

\begin{RE}\label{fp-cc}
Take an integer $r\geq 1$, a set $S\subseteq \{0,1\}^r$ and let $\mathcal{C}:=\cuboid(S)$. Let $y\in \mathbb{R}^{\mathcal{C}}_+$ and define $\alpha\in \mathbb{R}^{S}_+$ as follows: for every point $p\in S$ and corresponding member $C\in \mathcal{C}$, let $\alpha_p := \frac12\cdot y_C$. Then $y$ is a fractional packing of $\mathcal{C}$ of value two if, and only if, ${\bf 1}^\top \alpha = 1$ and $\sum_{p\in S} \alpha_p \cdot p = \frac12 \cdot {\bf 1}$. In particular, $\mathcal{C}$ has a fractional packing of value two if, and only if, $\frac12\cdot {\bf 1}\in \conv(S)$.
\end{RE}

\paragraph{Recursive construction of fractional packings.} Let $\mathcal{C}$ be a clean tangled clutter where $G(\mathcal{C})$ is not connected, and let $\{U,U'\}$ be the bipartition of a connected component of $G(\mathcal{C})$. Let $\mathcal{C}':=\mathcal{C}\setminus U/U'$. Observe that every member of $\mathcal{C}$ disjoint from $U$ contains $U'$, implying in turn that $\mathcal{C}'=\{C-U':C\in \mathcal{C},C\cap U=\emptyset\}$. Observe further that $\mathcal{C}'$ is a clean tangled clutter by \Cref{bipartite}, so it has a fractional packing of value two by \Cref{dense}. These observations are used to set up the following lemma:

\begin{LE}\label{fp-recursive}
Let $\mathcal{C}$ be a clean tangled clutter, where $G(\mathcal{C})$ is not connected. Let $\{U,U'\}$ be the bipartition of a connected component of $G(\mathcal{C})$, and let $z,z'$ be fractional packings of $\mathcal{C}\setminus U/U',\mathcal{C}/ U\setminus U'$ of value two, respectively. Let $y,y'\in \mathbb{R}^{\mathcal{C}}_+$ be defined as follows: $$
y_C:=\left\{
\begin{array}{ll}
z_{C-U'} ~\quad\text{ if } C\cap U=\emptyset\\
0 ~\quad\quad\quad\text{ otherwise}
\end{array}\right.
\quad\text{and}\quad
y'_C:=\left\{
\begin{array}{ll}
z'_{C-U} ~\quad\text{ if } C\cap U'=\emptyset\\
0 ~\quad\quad\quad\text{ otherwise.}
\end{array}\right.
$$ 
Then $\frac12 y + \frac12 y'$ is a fractional packing of $\mathcal{C}$ of value two. In particular, $\core{\mathcal{C}\setminus U/U'}\subseteq \core{\mathcal{C}}\setminus U/U'$.
\end{LE}
\begin{proof} 
We leave this as an exercise for the reader.
\end{proof}

\paragraph{Duplicated elements of the core.} For the next lemma, we need the following remark:

\begin{RE}\label{core-h2}
The core of any clean tangled clutter has covering number two.
\end{RE}

\begin{LE}\label{setcore-unique}
Let $\mathcal{C}$ be a clean tangled clutter over ground set $V$. Then two elements $u,v\in V$ are duplicates in $\core{\mathcal{C}}$ if, and only if, $u,v$ belong to the same part of the bipartition of a connected component of $G(\mathcal{C})$.
\end{LE}
\begin{proof}
$(\Leftarrow)$ follows \Cref{core->cuboid}. $(\Rightarrow)$ By \Cref{core->cuboid}, it suffices to show that $u,v$ belong to the same connected component of $G:=G(\mathcal{C})$. Suppose otherwise. In particular, $G$ is not connected. Let $\{U,U'\}$ be the bipartition of the connected component containing $u$ where $u\in U'$. Then $\mathcal{C}\setminus U/U'$ is a clean tangled clutter by \Cref{bipartite}. Let $w$ be a neighbour of $u$ in $G$; so $w\in U$. Then $\{w,u\}$ is a cover of $\mathcal{C}$. As every member of $\core{\mathcal{C}}$ containing $u$ also contains $v$, it follows that $\{w,v\}$ is a cover of $\core{\mathcal{C}}$, implying in turn that $\core{\mathcal{C}}\setminus U/U'$ has $\{v\}$ as a cover. However, $\core{\mathcal{C}\setminus U/U'}\subseteq \core{\mathcal{C}}\setminus U/U'$ by \Cref{fp-recursive}, so $\core{\mathcal{C}\setminus U/U'}$ has a cover of cardinality one, a contradiction to \Cref{core-h2}.
\end{proof}

\subsection{Proof of \Cref{setcore}}\label{sec:core/setcore}

Let $\mathcal{C}$ be a clean tangled clutter over ground set $V$, let $G:=G(\mathcal{C})$, and let $r:=\rank{\mathcal{C}}$. Recall that $r$ is the number of connected components of $G$. For each $i\in [r]$, let $\{U_i,V_i\}$ be the bipartition of the $i\textsuperscript{th}$ connected component of~$G$. Let $S:=\setcore{\mathcal{C}:U_1,V_1;U_2,V_2;\ldots;U_r,V_r}$. We claim that $S$ satisfies (i)-(iii) of \Cref{setcore}, thereby finishing the proof. It follows from the construction of $S$ that $\core{\mathcal{C}}$ is a duplication of $\cuboid(S)$. Moreover, by \Cref{setcore-unique}, $S$ is the unique set satisfying this property, up to isomorphism. Thus, {\bf (i)} holds. {\bf (ii)} Every fractional packing of value two in $\mathcal{C}$ corresponds to a fractional packing of value two in $\core{\mathcal{C}}$, and every fractional packing of value two in $\core{\mathcal{C}}$ corresponds to a fractional packing of value two in $\cuboid(S)$. By \Cref{fp-cc}, the fractional packings of value two in $\cuboid(S)$ are in correspondence with the different ways to express $\frac12 \cdot {\bf 1}$ as a convex combination of the points in $S$. These observations prove (ii). 

\setcounter{claim_nb}{0}
 
\begin{claim} 
$\frac12 \cdot {\bf 1}\in \conv(S)$.
\end{claim}
\begin{cproof}
By \Cref{dense}, $\mathcal{C}$ and therefore $\core{\mathcal{C}}$ has a fractional packing of value two, implying that $\cuboid(S)$ has a fractional packing of value two. It therefore follows from \Cref{fp-cc} that $\frac{1}{2} \cdot \mathbf{1}$ can be expressed as a convex combination of the points in $S$, thereby proving the claim.
\end{cproof}

\begin{claim} 
$\frac12 \cdot {\bf 1} \pm \frac12 \cdot e_i \in \conv(S)$ for each $i\in [r]$.
\end{claim}
\begin{cproof}
When $r = 1$, note that Claim~1 implies that $S = \{ 0, 1 \}$, so Claim~2 holds. Now assume $r\geq 2$.
Let $\mathcal{C}':=\mathcal{C}\setminus U_i/V_i$. Then $\mathcal{C}'$ is a clean tangled clutter by \Cref{bipartite}, and $\core{\mathcal{C}'}\subseteq \core{\mathcal{C}}\setminus U_i/V_i$ by \Cref{fp-recursive}. Let $z$ be a fractional packing of $\mathcal{C}'$ of value two. Then by \Cref{CS-conditions-2}, $$
\sum \left(z_{C'} : v\in C'\in \mathcal{C}'\right) = 1 \quad \forall~v\in V-(U_i\cup V_i).
$$ Define $y\in \mathbb{R}^{\mathcal{C}}_+$ as follows: $$
y_C:=\left\{
\begin{array}{ll}
z_{C-V_i} ~\quad\text{ if } C\cap U_i=\emptyset\\
0 ~\quad\quad\quad\text{ otherwise}.
\end{array}\right.
$$ Notice that \begin{align*}
{\bf 1}^\top y &= 2\\
\sum \left(y_{C} : v\in C\in \mathcal{C}\right) &= 1 \quad \forall~v\in V-(U_i\cup V_i)\\
\sum \left(y_{C} : v\in C\in \mathcal{C}\right) &= 2 \quad \forall~v\in V_i\\
\sum \left(y_{C} : v\in C\in \mathcal{C}\right) &= 0 \quad \forall~v\in U_i.
\end{align*} As $\supp(z)\subseteq \core{\mathcal{C}'}\subseteq \core{\mathcal{C}}\setminus U_i/V_i$, it follows that $\supp(y)\subseteq \core{\mathcal{C}}$. Define $\alpha\in \mathbb{R}^{S}_+$ as follows: for every point $p\in S$ and corresponding member $C\in \core{\mathcal{C}}$, let $\alpha_p := \frac12\cdot y_C$. Then the equalities above show that ${\bf 1}^\top \alpha = 1$ and $\sum_{p\in S} \alpha_p \cdot p = \frac12 \cdot {\bf 1} + \frac12\cdot e_i$. In particular, $\frac12 \cdot {\bf 1} + \frac12 \cdot e_i \in \conv(S)$. Repeating the argument on $\mathcal{C}/ U_i\setminus V_i$ yields $\frac12 \cdot {\bf 1} - \frac12 \cdot e_i \in \conv(S)$, thereby proving the claim.
\end{cproof}

Claims~1 and ~2 together imply that $\conv(S)$ is a full-dimensional polytope containing $\frac12\cdot {\bf 1}$ in its interior, so {\bf (iii)} holds. This finishes the proof of \Cref{setcore}.
\qed

\subsection{Applications}\label{sec:setcore/apps}

As the first application of \Cref{setcore}, we give the following characterization of the core of a clean tangled clutter. Note that this result is the converse of \Cref{core->cuboid}.

\begin{theorem}\label{core->cuboid-converse}
Let $\mathcal{C}$ be a clean tangled clutter of rank $r$, and for each $i\in [r]$, denote by $\{U_i,V_i\}$ the bipartition of the $i\textsuperscript{th}$ connected component of $G(\mathcal{C})$. Then $$\core{\mathcal{C}} = \{C\in \mathcal{C} : C\cap (U_i\cup V_i) = U_i \text{ or } V_i \quad \forall~i\in [r]\}.$$
\end{theorem}
\begin{proof}
Denote by $\mathcal{C}'$ the clutter on the right-hand side. Let $S:=\setcore{\mathcal{C}:U_1,V_1;\ldots;U_r,V_r}$. Let $S'$ be the subset of $\{0,1\}^r$ defined as follows: start with $S'=\emptyset$, and for each $C\in \mathcal{C}'$, add a point $p$ to $S'$ such that $$p_i=0 \quad\Leftrightarrow\quad C\cap (U_i\cup V_i) = U_i \quad \forall~i\in [r].$$ By \Cref{core->cuboid}, $$\core{\mathcal{C}} \subseteq \mathcal{C}'$$ so $S\subseteq S'$. We know from \Cref{setcore}~(iii) that $\frac12 \cdot {\bf 1}$ lies in the interior of $\conv(S)$, so $\frac12 \cdot {\bf 1}$ lies in the interior of $\conv(S')$. As a result, for every $p\in S'$, $\frac12\cdot {\bf 1}$ can be written as a convex combination of the points in $S'$ such that the coefficient of $p$ is nonzero. That is, by \Cref{fp-cc}, for each $C\in \mathcal{C}'$, there is a fractional packing of $\mathcal{C}'$ whose support includes $C$. As every fractional packing of $\mathcal{C}'$ is also a fractional packing of $\mathcal{C}$, it follows that $$\mathcal{C}'\subseteq \core{\mathcal{C}}$$ thereby finishing the proof of \Cref{core->cuboid-converse}.
\end{proof}

For the next application, we give an explicit description of the setcore when the rank is small:

\begin{theorem}\label{small-rank}
Let $\mathcal{C}$ be a clean tangled clutter with rank $r$. For each $i\in [r]$, denote by $\{U_i,V_i\}$ the bipartition of the $i\textsuperscript{th}$ connected component of $G(\mathcal{C})$. Then the following statements hold: \begin{enumerate}[(i)]
\item If $r=1$, then $\setcore{\mathcal{C}}=\{0,1\}$, and so $\core{\mathcal{C}}=\{U_1,V_1\}$.
\item If $r=2$, then $\setcore{\mathcal{C}}=\{00,10,01,11\}$, and so $\core{\mathcal{C}}=\{U_1\cup U_2,U_1\cup V_2,V_1\cup U_2,V_1\cup V_2\}$.
\item If $r=3$ and $\mathcal{C}$ does not have disjoint members, then 
$$\setcore{\mathcal{C}} = \{000,110,101,011\} \text{ or } \{100,010,001,111\}$$ and so \begin{align*}
\core{\mathcal{C}} = &\{U_1\cup U_2\cup U_3,U_1\cup V_2\cup V_3,V_1\cup U_2\cup V_3,V_1\cup V_2\cup U_3\}\\
 &\text{ or }\{U_1\cup U_2\cup V_3,U_1\cup V_2\cup U_3,V_1\cup U_2\cup U_3,V_1\cup V_2\cup V_3\}
 \end{align*}
\end{enumerate}
\end{theorem}
\begin{proof}
Let $S:=\setcore{\mathcal{C}}\subseteq \{0,1\}^r$. By \Cref{setcore}~(iii), $\frac12\cdot {\bf 1}$ lies in the interior of $\conv(S)$. This immediately {\bf (i)} and {\bf (ii)}. If $\mathcal{C}$ does not have disjoint members, then neither does $\core{\mathcal{C}}$, implying that $S$ does not have antipodal points. These two facts imply {\bf (iii)}.
\end{proof}

Finally, \Cref{setcore} allows us to restate the assumption that the setcore of a clean tangled clutter has a simplicial convex hull. This restatement is crucial for the proof of \Cref{main-PG}.

\begin{theorem}\label{unique-simplex}
Let $\mathcal{C}$ be a clean tangled clutter. Then $\conv(\setcore{\mathcal{C}})$ is a simplex if, and only if, $\mathcal{C}$ has a unique fractional packing of value two.
\end{theorem}
\begin{proof}
Let $S:=\setcore{\mathcal{C}}$. \Cref{setcore}~(ii) states that there is a one-to-one correspondence between the fractional packings of value two in $\mathcal{C}$ and the different ways to describe $\frac12\cdot{\bf 1}$ as a convex combination of the points in $S$. As a consequence, $\mathcal{C}$ has a unique fractional packing of value two if, and only if, $\frac12\cdot {\bf 1}$ can be written as a unique combination of the points in $S$. Since $\frac12\cdot {\bf 1}$ lies in the interior of $\conv(S)$ by \Cref{setcore}~(iii), the theorem follows.
\end{proof}

\section{From projective geometries to simplices}\label{sec:PG->simplex}

In this section, we show that the cocycle space of every projective geometry forms a simplex, and then prove \Cref{main-PG} $(\Leftarrow)$ as an immediate consequence.

\subsection{A primer on binary spaces and binary matroids}\label{sec:PG->simplex/primer}

Take an integer $n\geq 1$ and a set $S\subseteq \{0,1\}^n$. We say that $S$ is an \emph{affine vector space over $GF(2)$}, or simply an \emph{affine binary space}, if $a\tr b\tr c \in S$ for all points $a,b,c\in S$. If $S$ contains~${\bf 0}$, then $S$ is called a \emph{binary space}. Binary spaces enjoy the following transitive property:

\begin{RE}\label{transitive}
$S = S\tr a$ for every binary space $S$ and every point $a\in S$.
\end{RE}

Suppose $S$ is a binary space. By Basic Linear Algebra, there is a $0-1$ matrix $A$ with $n$ columns such that $S = \{x:Ax\equiv {\bf 0} \pmod{2}\}$. Let $M$ be the binary matroid over ground set $EM:=[n]$ that is represented by $A$. The {\it cycle space of $M$} is the set $\cycle{M}:=S$ and the \emph{cocycle space of $M$}, denoted $\cocycle{M}\subseteq \{0,1\}^n$, is the row space of $A$ over $GF(2)$. Notice that $\cycle{M},\cocycle{M}$ are binary spaces that are orthogonal complements. Observe that the binary matroid $M$ can be fully determined by either $A$, its cycle space or its cocycle space.

 A \emph{cycle of $M$} is a subset $C\subseteq EM$ such that $\chi_C\in \cycle{M}$, and a \emph{cocycle of $M$} is a subset $D\subseteq EM$ such that $\chi_D\in \cocycle{M}$. In particular, $\emptyset$ is both a cycle and a cocycle. Notice that every cycle and every cocycle have an even number of elements in common. A \emph{circuit of $M$} is a nonempty cycle that does not contain another nonempty cycle, and a \emph{cocircuit of $M$} is a nonempty cocycle that does not contain another nonempty cocycle. It is well-known that every cycle is either empty or the disjoint union of some circuits, and that every cocycle is either empty or the disjoint union of some cocircuits~\cite{Oxley11}. Observe that the cycles, circuits, cocycles, and cocircuits of $M$ correspond respectively to the cocycles, cocircuits, cycles, and circuits of the dual matroid~$M^\star$.

An element $e\in EM$ is a \emph{loop of $M$} if $\{e\}$ is a circuit, and it is a \emph{coloop of $M$} if $\{e\}$ is a cocircuit. Two distinct elements $e,f\in EM$ are \emph{parallel in $M$} if $\{e,f\}$ is a circuit. $M$ is a \emph{simple} binary matroid if it has no loops and no parallel elements, i.e. if every circuit has cardinality at least three. A \emph{triangle in $M$} is a circuit of cardinality three.

\begin{RE}\label{coloop}
Take an integer $n\geq 1$ and a binary space $S\subseteq \{0,1\}^n$, and let $M$ be the binary matroid whose cycle space is $S$. Then the points in $S$ do not agree on a coordinate if, and only if, $M$ has no coloops. Moreover, if the points in $S$ do not agree on a coordinate, then $|S\cap \{x:x_i=0\}| = |S\cap \{x:x_i=1\}|$ for all $i\in [n]$.
\end{RE}

\subsection{Proof of \Cref{main-PG} $(\Leftarrow)$}\label{sec:PG->simplex/proof}

Take an integer $k\geq 1$, and let $A$ be the $k\times (2^k-1)$ matrix whose columns are all the distinct $0-1$ vectors of dimension~$k$ that are nonzero. Recall that $PG(k-1,2)$ is the binary matroid represented by $A$, $\cycle{PG(k-1,2)} = \{x:Ax\equiv {\bf 0} \pmod{2}\}$ and $\cocycle{PG(k-1,2)}$ is the row space of $A$ generated over $GF(2)$. Recall further that $|\cocycle{PG(k-1,2)}|=2^k$. As $A$ has no zero column, and no two columns of it are equal, $PG(k-1,2)$ is a simple binary matroid. In particular, the points in $\cocycle{PG(k-1,2)}$ do not agree on a coordinate by \Cref{coloop}.

\begin{PR}\label{PG}
Take an integer $k\geq 2$. Then the following statements hold for $PG(k-1,2)$: \begin{enumerate}[(i)]
\item every nonempty cocycle has cardinality $2^{k-1}$,
\item every two elements appear together in a triangle,
\item every cycle is the symmetric difference of some triangles.
\end{enumerate}
\end{PR}
\begin{proof}
{\bf (i)} Let $D$ be a nonempty cocycle. Then $\chi_D$ is nonzero and belongs to $\cocycle{PG(k-1,2)}$. Let $A'$ be a $k\times (2^k-1)$ matrix with $0-1$ entries whose first row is $\chi_D$ and whose rows form a basis for $\cocycle{PG(k-1,2)}$ over $GF(2)$. Notice that the orthogonal complement of $\cocycle{PG(k-1,2)}$ over $GF(2)$ is equal to $$\cycle{PG(k-1,2)}=\{x:A'x\equiv {\bf 0} \pmod{2}\}.$$ As $PG(k-1,2)$ is a simple binary matroid, it follows that $A'$ has no zero column, and no two columns of it are equal. As $A'$ has $2^k-1$ columns and $k$ rows, it follows that the columns of $A'$ are all the $0-1$ vectors of dimension $k$ that are nonzero. In particular, every row of $A'$ has $2^{k-1}$ ones and $2^{k-1}-1$ zeros. In particular, $|D|=2^{k-1}$.

{\bf (ii)} Let $A$ be the $k\times (2^k-1)$ matrix representation of $PG(k-1,2)$ whose columns are all the $0-1$ vectors of dimension $k$ that are nonzero. Pick distinct elements $e,f$ of $PG(k-1,2)$, and let $a,b$ be the corresponding columns of $A$. Notice that $a+b \pmod{2}$ is another column of $A$, and let $g$ be the corresponding element of $PG(k-1,2)$. Then $\{e,f,g\}$ is the desired triangle of $PG(k-1,2)$.

{\bf (iii)} Let $C$ be a nonempty cycle. We proceed by induction on $|C|\geq 3$. The base $|C|=3$ holds trivially. For the induction step assume that $|C|\geq 4$. Pick distinct elements $e,f\in C$. By (ii) there is an element $g$ such that $\{e,f,g\}$ is a triangle. Since $C\triangle \{e,f,g\}$ is a cycle of smaller cardinality than $C$, the induction hypothesis applies and tell us that $C\triangle \{e,f,g\}$ is the symmetric difference of some triangles, implying in turn that $C$ is the symmetric difference of some triangles, thereby completing the induction step.
\end{proof}

We are now ready to present the key result of this subsection:

\begin{theorem}\label{PG-simplex}
Take an integer $k\geq 1$ and let $S:=\cocycle{PG(k-1,2)}$. Then $\conv(S)$ is a full-dimensional simplex containing $\frac12\cdot {\bf 1}$ in its interior. In particular, $\frac{1}{2^{k-1}} \cdot {\bf 1}$ is the unique fractional packing of $\cuboid(S)$ of value two.
\end{theorem}
\begin{proof}
Let $n:=2^k-1$. We know that $S$ is a subset of $\{0,1\}^n$ and has exactly $n+1$ points. It follows from \Cref{PG}~(i) that the inequality $$\sum_{i=1}^n x_i \leq \frac{n+1}{2}$$ is valid for $\conv(S)$, and that every point in $S$ except for ${\bf 0}$ satisfies this inequality at equality. As $S$ is a binary space, $S\tr p = S$ for every point $p\in S$ by~\Cref{transitive}. This transitive property implies that for each $p\in S$, the transformed inequality $$\sum_{i : p_i=0} x_i + \sum_{j : p_j =1} (1-x_j) \leq \frac{n+1}{2}$$ is also valid for $\conv(S)$, and every point in $S$ except for $p$ satisfies this inequality at equality. Hence, $\conv(S)$ is an $n$-dimensional simplex whose $n+1$ facets are as described above. 

As the point $\frac12\cdot {\bf 1}$ satisfies every inequality strictly, it lies in the interior of $\conv(S)$. In fact, as $S$ is a binary space whose points do not agree on a coordinate, $|S\cap \{x:x_i=0\}| = |S\cap \{x:x_i=1\}|$ for each $i\in [n]$ by~\Cref{coloop}, so $$\sum_{p\in S} \frac{1}{n+1} \cdot p = \frac12 \cdot {\bf 1}.$$ As $\conv(S)$ is a simplex, it follows from \Cref{fp-cc} that $\frac{2}{n+1}\cdot {\bf 1} = \frac{1}{2^{k-1}} \cdot {\bf 1}$ is the unique fractional packing of $\cuboid(S)$ of value two, thereby finishing the proof of \Cref{PG-simplex}.
\end{proof}

As a consequence,

\begin{proof}[Proof of \Cref{main-PG} $(\Leftarrow)$]
Let $\mathcal{C}$ be a clean tangled clutter. If $\setcore{\mathcal{C}}$ is the cocycle space of a projective geometry, then \Cref{PG-simplex} implies that $\conv(\setcore{\mathcal{C}})$ is a simplex, as required.
\end{proof}

\section{From simplices to projective geometries}\label{sec:simplex->PG}

In this section, after presenting a lemma on constructing projective geometries, we prove \Cref{main-PG}~$(\Rightarrow)$, and then present an appealing consequence characterizing when a simplex comes from a projective geometry. 

We start with the following key lemma allowing for an inductive argument:

\begin{LE}\label{fp-recursive-2}
Let $\mathcal{C}$ be a clean tangled clutter with a unique fractional packing of value two. Suppose $G(\mathcal{C})$ is not connected, and let $\{U,U'\}$ be the bipartition of some connected component of $G(\mathcal{C})$. Then $\mathcal{C}\setminus U/U'$ is a clean tangled clutter with a unique fractional packing of value two. Moreover, if $y,z$ are the fractional packings of $\mathcal{C},\mathcal{C}\setminus U/U'$ of value two, respectively, then $\supp(z) = \supp(y)\setminus U/U'$.
\end{LE}
\begin{proof}
By \Cref{bipartite}, $\mathcal{C}\setminus U/U'$ and $\mathcal{C}/U\setminus U'$ are clean tangled clutters, so we may apply \Cref{dense} and conclude that they have fractional packings $z,z'$ of value two, respectively. Let $t,t'\in \mathbb{R}^{\mathcal{C}}_+$ be defined as follows: $$
t_C:=\left\{
\begin{array}{ll}
z_{C-U'} ~\quad\text{ if } C\cap U=\emptyset\\
0 ~\quad\quad\quad\text{ otherwise}
\end{array}\right.
\quad\text{and}\quad
t'_C:=\left\{
\begin{array}{ll}
z'_{C-U} ~\quad\text{ if } C\cap U'=\emptyset\\
0 ~\quad\quad\quad\text{ otherwise.}
\end{array}\right.
$$ By \Cref{fp-recursive}, $\frac12 t + \frac12 t'$ is a fractional packing of $\mathcal{C}$ of value two. It therefore follows from the uniqueness assumption that $\frac12 t + \frac12 t' = y$. Subsequently, $z$ must be the unique fractional packing of $\mathcal{C}\setminus U/U'$ of value two, $z'$ must be the unique fractional packing of $\mathcal{C}/ U\setminus U'$ of value two, and \begin{align*}
\supp(z) &= \supp(y)\setminus U/U' \\
\supp(z') &= \supp(y)/ U\setminus U',\end{align*} as desired.~\end{proof}

\subsection{Constructing projective geometries}\label{sec:unique-dyadic/construction}

For an integer $r\geq 1$ and a set $S\subseteq \{0,1\}^r$, the {\it incidence matrix of $S$} is the matrix whose rows are the points in~$S$. Denote by $J$ the all-ones matrix of appropriate dimensions. Take an integer $k\geq 1$ and let $A$ be the incidence matrix of $\cocycle{PG(k-1,2)}$. Then every column of $A$ has $2^{k-1}$ ones and $2^{k-1}$ zeros. In fact, $A$ has the following recursive description:

\begin{RE}\label{PG-recursive}
Take an integer $k\geq 2$. If $A'$ is the incidence matrix of $\cocycle{PG(k-2,2)}$, then up to permuting rows and columns, $$
\begin{pmatrix}
{\bf 1}&A'&J-A'\\
{\bf 0}&A'&A'
\end{pmatrix}
$$ is the incidence matrix of $\cocycle{PG(k-1,2)}$. Moreover, every element of $PG(k-1,2)$ can be used as the left-most column in the incidence matrix above.
\end{RE}

Consequently, for every pair $a,b$ of columns of $A$, $$|\{j:a_j=b_j=0\}|=|\{j:a_j=b_j=1\}|=|\{j:a_j=1, b_j=0\}|=|\{j:a_j=0,b_j=1\}|=2^{k-2}.$$ Two columns of a $0-1$ matrix are {\it complementary} if they add up to the all-ones vector. If $\mathcal{C}$ is the cuboid of $\cocycle{PG(k-1,2)}$, then every column of $M(\mathcal{C})$ has $2^{k-1}$ ones, and by the expressions above, every pair of columns of $M(\mathcal{C})$ are either complementary or have exactly $2^{k-2}$ ones in common.

\begin{RE}\label{PG->PG}
Take an integer $k\geq 2$, and let $\mathcal{C}:=\cuboid(\cocycle{PG(k-1,2)})$. Then for every minimum cover $\{u,v\}$ of $\mathcal{C}$, the minor $\mathcal{C}\setminus u/v$ is obtained from $\cuboid(\cocycle{PG(k-2,2)})$ after duplicating every element once.
\end{RE}

This remark is an immediate consequence of \Cref{PG-recursive}, and is helpful to keep in mind when parsing the hypotheses of the following lemma, which is the main result of this section:

\begin{LE}\label{unique-lemma}
Take an integer $r\geq 2$ and a clutter $\mathcal{C}$ whose ground set $V$ is partitioned into nonempty parts $U_1,V_1,\ldots,U_r,V_r$ such that \begin{itemize}
\item the elements in each part are duplicates, 
\item for each $i\in [r]$, if $u\in U_i$ and $v\in V_i$, then $\{u,v\}$ is a transversal of $\mathcal{C}$, and
\item for each $i\in [r]$, $\mathcal{C}\setminus U_i/V_i$ (resp. $\mathcal{C}/U_i\setminus V_i$) is a duplication of the cuboid of the cocycle space of a projective geometry.
\end{itemize}
Assume further that $\mathcal{C}$ has exactly $r+1$ members and a unique fractional packing of value two. Then there is an integer $k\geq 2$ such that $r=2^k-1$ and $\mathcal{C}$ is a duplication of $\cuboid(\cocycle{PG(k-1,2)})$.
\end{LE}
\begin{proof}
We may assume after contracting some duplicate elements that $U_i=\{u_i\}$ and $V_i=\{v_i\}$ for each $i\in [r]$. In particular, $\mathcal{C}$ is a cuboid. As $\mathcal{C}$ has a fractional packing of value two, it follows that $\tau(\mathcal{C})\geq 2$, so $\mathcal{C}$ is a tangled clutter. For each $i\in [r]$, let $f(u_i) := v_i$ and $f(v_i) := u_i$.

\begin{claim} 
$\mathcal{C}$ does not have duplicated elements. In particular, if $\{u,v\}$ is a transversal of $\mathcal{C}$, then $v=f(u)$.
\end{claim}
\begin{cproof}
Suppose for a contradiction that $u,u'$ are duplicates. Since $\tau(\mathcal{C})= 2$, $\{u,u'\}$ is not a cover, so $u'\neq f(u)$. But then $\mathcal{C}\setminus f(u)/u$ has $\{u'\}$ as a cover, a contradiction as $\mathcal{C}\setminus f(u)/u$ is a duplication of the cuboid of the cocycle space of a projective geometry.
\end{cproof}

In what follows the reader should keep in mind that our labeling of the columns of $M(\mathcal{C})$ induces a labeling for the columns of $M(\mathcal{C}\setminus f(u)/u)$, for each $u\in V$. In particular, $M(\mathcal{C}\setminus f(u)/u)$ and $M(\mathcal{C}/f(u)\setminus u)$ have the same column labels, for each $u\in V$.

\begin{claim} 
There is an integer $k\geq 2$ such that the following statements hold: \begin{enumerate}[(1)]
\item for each $u\in V$, $\mathcal{C}\setminus f(u)/u$ is a duplication of $\cuboid(\cocycle{PG(k-2,2)})$,
\item $|\mathcal{C}|=2^k$,
\item every column of $M(\mathcal{C})$ has exactly $2^{k-1}$ ones,
\item every pair of columns of $M(\mathcal{C})$ are either complementary or have exactly $2^{k-2}$ ones in common.
\end{enumerate}
\end{claim}
\begin{cproof}
For each $u\in V$, $\mathcal{C}\setminus f(u)/u$ is a duplication of $\cuboid(\cocycle{PG(k_u-2,2)})$ for some integer $k_u\geq 2$. In particular, every column $u$ of $M(\mathcal{C})$ has exactly $2^{k_u-1} = |\cocycle{PG(k_u-2,2)}|$ ones. Notice now that if $u\in V$ and $w\in V-\{u,f(u)\}$, then the number of ones in column $w$ of $M(\mathcal{C})$ is equal to the sum of the number of ones in column $w$ of $M(\mathcal{C}\setminus f(u)/u)$ and the number of ones in column $w$ of $M(\mathcal{C}/f(u)\setminus u)$, so $$2^{k_w-1} = 2^{k_u-2} + 2^{k_{f(u)}-2}$$ implying in turn that $k_w=k_u=k_{f(u)}$. As a result, $(k_u:u\in V)$ are all equal to $k$ for some integer $k\geq 2$. It can be readily checked that (1)-(4) hold for $k$, as required.
\end{cproof}

Following up on Claim~2~(4), if a pair of columns of $M(\mathcal{C})$ are complementary, then by Claim~1, the column labels must be $u,f(u)$ for some $u\in V$.

\begin{claim} 
$\frac{1}{2^{k-1}}\cdot {\bf 1}\in \mathbb{R}^{\mathcal{C}}$ is the unique fractional packing of $\mathcal{C}$ of value two.
\end{claim}
\begin{cproof}
This follows from Claim~2~(2)-(3) and our assumption that $\mathcal{C}$ has a unique fractional packing of value two.
\end{cproof}


\begin{claim} 
The following statements hold for every $u\in V$:
\begin{enumerate}
\item[(1)] a pair of identical columns in the matrix $M(\mathcal{C}\setminus f(u)/u)$ correspond to a complementary pair of columns in the matrix $M(\mathcal{C}/ f(u)\setminus u)$,
\item[(2)] $M(\mathcal{C}\setminus f(u)/u)$ does not have three identical columns, 
\item[(3)] $r=2^k-1$, and
\item[(4)] $\mathcal{C}\setminus f(u)/u$ is obtained from $\cuboid(\cocycle{PG(k-2,2)})$ after duplicating every element exactly once.
\end{enumerate}
\end{claim}
\begin{cproof}
{\bf (1)} follows from Claim~2~(4). 
{\bf (2)} follows from (1). 
{\bf (3)} follows from Claim~2~(2) and our assumption that $|\mathcal{C}|=r+1$.
{\bf (4)} Claim~2~(1), together with part (2) of this claim, implies that the minor $\mathcal{C}\setminus f(u)/u$ is obtained from $\cuboid(\cocycle{PG(k-2,2)})$ after duplicating every element at most once. In particular, $$
2r = |V| \leq 2+2\cdot 2\cdot (2^{k-1}-1) = 2\cdot (2^{k}-1).$$ However, $r = 2^k-1$ by part (3) of this claim, so equality must hold throughout the above inequalities, thereby proving (4).
\end{cproof}

Pick $S\subseteq \{0,1\}^r$ containing ${\bf 0}$ such that $\mathcal{C}=\cuboid(S)$. We prove that $S=\cocycle{PG(k-1,2)}$. Denote by $A$ the incidence matrix of $S$. Notice that $A$ is a column submatrix of $M(\mathcal{C})$, and the column labels of $A$ form a subset of $V$ and a transversal of $\{\{u,f(u)\}:u\in V\}$.

\begin{claim} 
In $A$ the sum of every two columns modulo $2$ is equal to another column.
\end{claim}
\begin{cproof}
Pick two columns of $A$ with column labels $u,w\in V$. By Claim~4~(4), in $M(\mathcal{C}\setminus f(u)/u)$, column $w$ is identical to another column $v$. Notice that $v\in V-\{u,f(u),w,f(w)\}$. By Claim~4~(1), in $M(\mathcal{C}/f(u)\setminus u)$, columns $w,v$ are complementary. Thus, in $M(\mathcal{C})$, columns $u,w,v$ add up to ${\bf 1}$ modulo $2$, implying in turn that columns $u,w,f(v)$ add up to ${\bf 0}$ modulo $2$. We know that columns $u,w$ of $M(\mathcal{C})$ are also present in $A$, and that exactly one of the columns $v,f(v)$ of $M(\mathcal{C})$ is present in $A$. As ${\bf 0}\in S$, $A$ has a zero row, so no three of its columns can add up to ${\bf 1}$ modulo $2$, implying in turn that $f(v)$ must be a column of~$A$ instead of $v$. As a result, in $A$, columns $u,w$ add up to column $f(v)$ modulo $2$, as required.
\end{cproof}

We next use \Cref{PG-recursive} to argue that up to permuting rows and columns, $A$ is the incidence matrix of $\cocycle{PG(k-1,2)}$. To this end, denote by $v_0\in V$ the label of the first column of $A$. For $j\in \{0,1\}$, denote by $I_j$ the rows of $A$ corresponding to $\{x\in S: x_{v_0}=j\}$. By Claim~2~(3), $|I_0|=|I_1|=2^{k-1}$. Notice that $\frac{r-1}{2} = 2^{k-1}-1$ by Claim~4~(3). Label the columns of $A$ other than $v_0$ as $v_1,u_1,v_2,u_2,$ $\ldots,v_{\frac{r-1}{2}},u_{\frac{r-1}{2}}$ where for each $i\in \left[\frac{r-1}{2}\right]$, the sum of columns $v_0$ and $v_i$ modulo $2$ is equal to column $u_i$ -- such a labeling exists because of Claim~5. Define matrices $A_1,A_2,A_3,A_4$: \begin{itemize}
\item $A_1$ is the $I_1\times \{v_1,\ldots,v_{\frac{r-1}{2}}\}$ submatrix of $A$,
\item $A_2$ is the $I_1\times \{u_1,\ldots,u_{\frac{r-1}{2}}\}$ submatrix of $A$,
\item $A_3$ is the $I_0\times \{v_1,\ldots,v_{\frac{r-1}{2}}\}$ submatrix of $A$,
\item $A_4$ is the $I_0\times \{u_1,\ldots,u_{\frac{r-1}{2}}\}$ submatrix of $A$.
\end{itemize} Then $A_3=A_4$ and $A_1+A_2 = J$. After swapping the labels $v_i$ and $u_i, i\in \left[\frac{r-1}{2}\right]$, if necessary, we may assume that $A_1$ has a zero row. Notice further that as ${\bf 0}\in S$ and $A_3=A_4$, the matrix $A_3$ also has a zero row. As a result, by Claim~4~(4), up to permuting rows and columns, the following three matrices are equal: $A_1$, $A_3$, and the incidence matrix of $\cocycle{PG(k-2,2)}$.

For the rest of the proof, we work with the projective geometry $PG(k-2,2)$ whose labeling agrees with the column labels of $A_3$, that is, the cocycles of the labeled $PG(k-2,2)$ are the rows of $A_3$.

\begin{claim} 
Up to permuting rows, $A_1$ and $A_3$ are equal.
\end{claim}
\begin{cproof}
This is obviously true if $k=2$. We may therefore assume that $k\geq 3$. It suffices to show that every row of $A_1$ is equal to some row of $A_3$, because the two matrices are already equal up to permuting rows and columns. Pick a row $\chi_D$ of $A_1$ for some $D\subseteq \{v_1,\ldots,v_{\frac{r-1}{2}}\}$. We need to show that $D$ is a cocycle of (the labeled) $PG(k-2,2)$. Pick a triangle $\{v_i,v_j,v_k\}$ of $PG(k-2,2)$, that is, the corresponding columns of $A_3$ add up to zero modulo $2$. Consider now the columns $v_i,v_j$ of $A$. By Claim~5, the sum of these two columns modulo $2$ is another column of $A$. This column is either $v_k$ or $u_k$, and in fact since $A_1$ has a zero row, it must be $v_k$. As a result, columns $v_i,v_j,v_k$ of $A_1$ also add up to zero modulo $2$, implying in turn that $|D\cap \{v_i,v_j,v_k\}|$ is even. Thus, $D$ intersects every triangle of $PG(k-2,2)$ an even number of times, so by \Cref{PG}~(iii), $D$ intersects every cycle of $PG(k-2,2)$ an even number of times, implying in turn that $D$ is a cocycle of $PG(k-2,2)$, as required.
\end{cproof}

We may therefore assume that $A_1=A_3$, implying in turn that $A_1=A_3=A_4$ and $A_2=J-A_1$. As $A_1$ is the incidence matrix of $\cocycle{PG(k-2,2)}$, it follows from \Cref{PG-recursive} that $A$ is the incidence matrix of $\cocycle{PG(k-1,2)}$, so $S=\cocycle{PG(k-1,2)}$. As $\mathcal{C}=\cuboid(S)$, and as $r=2^k-1$ by Claim~4~(3), we have finished the proof of \Cref{unique-lemma}.
\end{proof}

It is worth pointing out that the assumption $|\mathcal{C}|=r+1$ in \Cref{unique-lemma} can be removed without affecting the conclusion, but this comes at the expense of a much longer proof of Claim~4, parts (3) and (4), one that requires the notion of \emph{binary clutters}.

\subsection{Proof of \Cref{main-PG} $(\Rightarrow)$}\label{sec:unique-dyadic/main}

Let $\mathcal{C}$ be a clean tangled clutter over ground set $V$ whose setcore has a simplicial convex hull. By \Cref{unique-simplex}, $\mathcal{C}$ has a unique fractional packing $y$ of value two. We shall prove by induction on $|V|\geq 2$ that \begin{quote}
$(\star)$ there is an integer $k\geq 1$ such that $y$ is $\frac{1}{2^{k-1}}$-integral, $\rank{\mathcal{C}}=2^k-1$ and $\supp(y)$ is a duplication of $\cuboid(\cocycle{PG(k-1,2)})$.\end{quote} For the base case $|V|=2$, as $\mathcal{C}$ is tangled, it must consist of two members of size one each, so $(\star)$ holds for $k=1$. For the induction step, assume that $|V|\geq 3$. Let $r:=\rank{\mathcal{C}}$ and $S:=\setcore{\mathcal{C}}\subseteq \{0,1\}^r$. By \Cref{setcore}~(iii) and our assumption, $\conv(S)$ is a full-dimensional simplex, implying in turn that $|S|=r+1$. Let $G:=G(\mathcal{C})$, and for each $i\in [r]$, let $\{U_i,V_i\}$ be the bipartition of the $i\textsuperscript{th}$ connected component of~$G$. As $\supp(y)\subseteq \core{\mathcal{C}}$, \Cref{core->cuboid} implies Claim~1 below:

\setcounter{claim_nb}{0}

\begin{claim} 
For each $C\in \supp(y)$ and $i\in [r]$, $C\cap (U_i\cup V_i)$ is either $U_i$ or $V_i$.
\end{claim}

\begin{claim} 
If $r=1$, then $(\star)$ holds for $k=1$.
\end{claim}
\begin{cproof}
Assume that $r=1$. Then $\supp(y)\subseteq \{U_1,V_1\}$ by Claim~1, and as $\supp(y)$ contains a fractional packing of $\mathcal{C}$ of value two, we must have that $\supp(y)= \{U_1,V_1\}$, and the claim follows.
\end{cproof}

We may therefore assume that $r\geq 2$.

\begin{claim} 
The following statements hold:
\begin{enumerate}[(1)]
\item $|\supp(y)|=r+1$,
\item $\supp(y)$ has a unique fractional packing of value two,
\item the elements in each of $U_1,V_1,\ldots,U_r,V_r$ are duplicates in $\supp(y)$,
\item for each $i\in [r]$, if $u\in U_i$ and $v\in V_i$, then $\{u,v\}$ is a transversal of $\supp(y)$, and
\item for each $i\in [r]$, $\supp(y)\setminus U_i/V_i$ (resp. $\supp(y)/U_i\setminus V_i$) is a duplication of the cuboid of the cocycle space of a projective geometry.
\end{enumerate}
\end{claim}
\begin{cproof}
{\bf (1)} Since $y$ is the unique fractional packing of $\mathcal{C}$ of value two, we have $\core{\mathcal{C}} = \supp(y)$. Subsequently, $|\supp(y)| = |\core{\mathcal{C}}| = |S| = r+1$.
{\bf (2)} is obvious, and {\bf (3)} and {\bf (4)} follow from Claim~1. {\bf (5)} By \Cref{fp-recursive-2}, the minor $\mathcal{C}\setminus U_i/V_i$ is a clean tangled clutter with a unique fractional packing $z$ of value two, and $\supp(z) = \supp(y)\setminus U_i/V_i$. Our induction hypothesis applied to $\mathcal{C}\setminus U_i/V_i$ implies that $\supp(z)$, which is equal to $\supp(y)\setminus U_i/V_i$, is a duplication of the cuboid of the cocycle space of a projective geometry, as required. 
\end{cproof}

We may therefore apply \Cref{unique-lemma} to $\supp(y)$ to conclude that for some integer $k\geq 2$, $r=2^{k}-1$ and $\supp(y)$ is a duplication of $\cuboid(\cocycle{PG(k-1,2)})$. It follows from \Cref{PG-simplex} that $y$ assigns $\frac{1}{2^{k-1}}$ to the members of $\supp(y)$, so $(\star)$ holds. This completes the induction step. 

We have shown that $(\star)$ holds. As a consequence, $\core{\mathcal{C}}=\supp(y)$ is a duplication of the cuboid of $\cocycle{PG(k-1,2)}\subseteq \{0,1\}^r$. The uniqueness of the setcore (\Cref{setcore}~(i)) implies that $\setcore{\mathcal{C}}=\cocycle{PG(k-1,2)}$, thereby finishing the proof of \Cref{main-PG}~$(\Rightarrow)$.
\qed

\subsection{Binary clutters and an application}\label{sec:unique-dyadic/app}

A clutter $\mathcal{C}$ is \emph{binary} if the symmetric difference of any three members contains a member~\cite{Lehman64}. Observe that if a clutter is binary, then so is every duplication of it. It is known that $\mathcal{C}$ is a binary clutter if, and only if, $|C\cap B|\equiv 1 \pmod{2}$ for all $C\in \mathcal{C},B\in b(\mathcal{C})$~\cite{Lehman64}. In particular, a clutter is binary if and only if its blocker is binary. Observe that the deltas, extended odd holes and their blockers are not binary. If a clutter is binary, so is every minor of it~\cite{Seymour76}. Subsequently, 
 
\begin{RE}\label{binary->clean}
Every binary clutter is clean.
\end{RE}

Examples include the clutter of minimal $T$-joins of a graft, and the clutter of odd circuits of a signed graph (the ground set in each case is the edge set of the underlying graph)~\cite{Cornuejols01}. Another class of binary clutters comes from affine binary spaces. 

\begin{RE}\label{binaryspace-binaryclutter}
Take an integer $n\geq 1$ and a set $S\subseteq \{0,1\}^n$. Then $S$ is an affine binary space if, and only if, $\cuboid(S)$ is a binary clutter.
\end{RE}

We are now ready to prove the following appealing consequence of \Cref{main-PG}:

\begin{theorem}\label{PG-simplex-3}
Take an integer $n\geq 1$ and a set $S\subseteq \{0,1\}^n$ whose convex hull is a simplex containing $\frac12\cdot {\bf 1}$ in its relative interior. Then exactly one of the following statements holds: \begin{itemize}
\item $\cuboid(S)$ has a delta or the blocker of an extended odd hole minor, or
\item $S$ is a duplication of the cocycle space of a projective geometry over the two-element field.
\end{itemize} 
\end{theorem}
\begin{proof}
Let $\mathcal{C}:=\cuboid(S)$. If $S$ is a duplication of the cocycle space of a projective geometry, then up to twisting, $S$ is a binary space, so $\mathcal{C}$ is a binary clutter by \Cref{binaryspace-binaryclutter}, implying in turn that it is clean by \Cref{binary->clean}. Conversely, assume that $\mathcal{C}$ is clean. As $\conv(S)$ is a simplex containing $\frac12\cdot {\bf 1}$ in its relative interior, \begin{itemize}
\item the points in $S$ do not all agree on a coordinate, so $\mathcal{C}$ is tangled, and
\item by \Cref{fp-cc} on the connection between $\conv(S)$ and fractional packings of $\mathcal{C}$, $\mathcal{C}$ must have a unique fractional packing of value two, one whose support is $\mathcal{C}$.
\end{itemize} In particular, $\mathcal{C}=\core{\mathcal{C}}$, so $S$ is a duplication of $\setcore{\mathcal{C}}$. As $\conv(S)$ is a simplex, so is $\conv(\setcore{\mathcal{C}})$, so by \Cref{main-PG}, $\setcore{\mathcal{C}}$ is isomorphic to the cocycle space of a projective geometry, implying in turn that $S$ is a duplication of the cocycle space of a projective geometry, as required.
\end{proof}

\section{Finding the Fano plane as a minor}\label{sec:fano}

In this section, after presenting a few ingredients, we prove \Cref{main-L7}, and then prove a consequence of the result. 

\subsection{Monochromatic covers in clean tangled clutters}\label{sec:fano/mono}

Let $\mathcal{C}$ be a clean tangled clutter. A cover is \emph{monochromatic} if it is monochromatic in some (proper) bicoloring of the bipartite graph $G(\mathcal{C})$. In this subsection, we prove a lemma on monochromatic minimal covers in clean tangled clutters. We need the following result from Mathematical Logic:

\begin{PR}[\cite{Robinson65}]\label{resolution}
Take an integer $r\geq 1$ and a set $S\subseteq \{0,1\}^r$. Pick disjoint subsets $I,J\subseteq [r]$ and disjoint subsets $I',J'\subseteq [r]$ such that the following inequalities are valid $S$: \begin{align*}
\sum_{i\in I}x_i+\sum_{j\in J}(1-x_j) &\geq 1\\
\sum_{i\in I'}x_i+\sum_{j\in J'}(1-x_j) &\geq 1
\end{align*} If $k\in I\cap J'$, then the following inequality is also valid for $S$: $$
\sum_{i\in (I\cup I')-\{k\}}x_i+\sum_{j\in (J\cup J')-\{k\}}(1-x_j) \geq 1.
$$
\end{PR}
\begin{proof}
We leave the proof as an exercise for the reader.
\end{proof}

\Cref{resolution} is known as the \emph{Resolution Principle} and the derived inequality is referred to as the \emph{resolvent} of the other two inequalities. We use this remark to prove the following, a key ingredient needed for the proof of \Cref{main-L7}.

\begin{theorem}\label{mono}
Let $\mathcal{C}$ be a clean tangled clutter over ground set $V$ of rank $r$, and for each $i\in [r]$, denote by $\{U_i,V_i\}$ the bipartition of the $i\textsuperscript{th}$ connected component of $G:=G(\mathcal{C})$. Suppose for some integer $k\in [r]$ that $V_1\cup \cdots \cup V_k$ is a cover of $\mathcal{C}$. Then $k\geq 3$. Moreover, if $k=3$, then $V_1\cup V_2\cup V_3$ contains a minimal cover of cardinality three picking exactly one element from each $V_i,i\in [3]$.
\end{theorem}
\begin{proof}
Let $S:=\setcore{\mathcal{C}:U_1,V_1;U_2,V_2;\ldots;U_r,V_r}$. Since $V_1\cup \cdots \cup V_k$ is a cover of $\mathcal{C}$, it is also a cover of $\core{\mathcal{C}}$, so every member of $\core{\mathcal{C}}$ contains at least one of $V_1,\ldots,V_k$, by \Cref{core->cuboid}. In particular, the inequality $x_1+\cdots+x_k\geq 1$ is valid for $\conv(S)$. As $\frac12\cdot {\bf 1}$ lies in the interior of $\conv(S)$ by \Cref{setcore}~(iii), it follows that $k\geq 3$.

Assume that $k=3$. Let $B$ be a minimum cardinality cover contained in $V_1\cup V_2\cup V_3$. What we just showed implies that $B\cap V_i\neq \emptyset$ for $i\in [3]$. We claim that $|B\cap V_i|=1$ for each $i\in [3]$, thereby finishing the proof. Suppose otherwise. We may assume that $|B\cap V_3|\geq 2$. Let $I:=B-V_3$, $J:=V-(I\cup U_3\cup V_3)$, and $\mathcal{C}':=\mathcal{C}\setminus I/J$. Note that $\mathcal{C}'$ is a clean clutter and has ground set $U_3\cup V_3$. Notice further that every edge of $G[U_3\cup V_3]$ gives a cardinality-two cover of $\mathcal{C}'$. \begin{enumerate}
\item[Case 1:] $\tau(\mathcal{C}')= 2$. In this case, $\mathcal{C}'$ is a tangled clutter where $G[U_3\cup V_3]\subseteq G(\mathcal{C}')$. In particular, $G(\mathcal{C}')$ is a connected bipartite graph whose bipartition inevitably is $\{U_3,V_3\}$. Thus, $\rank{\mathcal{C}'}=1$, so $\core{\mathcal{C}'} = \{U_3,V_3\}$ by \Cref{small-rank}~(i). However, as $B$ is a cover of $\mathcal{C}$, $B-I=B\cap V_3$ is a cover of $\mathcal{C}'$, a contradiction as $B\cap V_3$ is disjoint from $U_3\in \core{\mathcal{C}'}\subseteq \mathcal{C}'$.

\item[Case 2:] $\tau(\mathcal{C}') \leq 1$. That is, there is a minimal cover $D$ of $\mathcal{C}$ such that $D\cap J=\emptyset$ and $|D-I|\leq 1$. As $D\cap I\subseteq I=B-V_3\subsetneq B$, and $B$ is a minimal cover of $\mathcal{C}$, it follows that $D\cap I$ is not a cover of $\mathcal{C}$, so $D-I\neq \emptyset$. Thus, $|D-I|=1$. Let $u$ be the element in $D-I\subseteq U_3\cup V_3$.

\begin{enumerate}
\item[Case 2.1:] $u\in U_3$. In this case, $V_1\cup V_2\cup U_3$ is a cover of $\mathcal{C}$, implying that the inequality $x_1+x_2+(1-x_3)\geq 1$ is valid for $S$. However, $V_1\cup V_2\cup V_3$ is also a cover of $\mathcal{C}$, so $x_1+x_2+x_3\geq 1$ is valid for $S$, too. By applying the Resolution Principle, \Cref{resolution}, we get that $x_1+x_2\geq 1$ is also valid for $S$. However, $\frac12\cdot {\bf 1}$ lies in the interior of $\conv(S)$ by \Cref{setcore}~(iii), a contradiction.

\item[Case 2.2:] $u\in V_3$. In this case, $$|D| = |D\cap I| + |D-I| = |B-V_3|+1<|B-V_3|+|B\cap V_3|=|B|,$$ where the strict inequality follows from our contrary assumption that $|B\cap V_3|\geq 2$. However, $|D|<|B|$ contradicts our minimal choice of $B$ as the minimum cover of $\mathcal{C}$ contained in $V_1\cup V_2\cup V_3$.
\end{enumerate}
\end{enumerate}

We obtained a contradiction in each case, as desired.
\end{proof}

\subsection{A lemma for finding an $\mathbb{L}_7$ minor}\label{sec:fano/lemma}

Recall that $$\mathbb{L}_7=\{\{1,2,3\},\{1,4,5\},\{1,6,7\},\{2,5,6\},\{2,4,7\},\{3,4,6\},\{3,5,7\}\}$$ and $b(\mathbb{L}_7)=\mathbb{L}_7$. This clutter enjoys a lot of symmetries. $\mathbb{L}_7$ has an automorphism mapping every element to every other element, and an automorphism mapping every member to every other member. These facts are crucial throughout this subsection.

\begin{RE}\label{trees}
Let $G=(V,E)$ be a connected, bipartite graph with bipartition $\{U,U'\}$ where $U,U'\neq \emptyset$. Assume that there exists a subset $X\subseteq U'$ such that $2\leq |X|\leq 3$, and there is no proper vertex-induced subgraph that is connected and contains $X$. Then $G$ is a tree whose leaves are in $X$.
\end{RE}
\begin{proof}
By our minimality assumption, every vertex in $V-X$ is a cut-vertex of $G$ separating at least two vertices in $X$. We claim that $G$ is a tree. Suppose otherwise. Then there is a circuit $C\subseteq V$. For every vertex $v\in C$, there is a vertex $g(v)\in X$ such that \begin{itemize}
\item if $v\in X$, then $g(v)=v$,
\item otherwise, $g(v)$ is a vertex of $X$ such that every path between it and $C-\{v\}$ includes $v$.
\end{itemize} 
Notice that if $v,v'$ are distinct vertices of $C$, then $g(v)\neq g(v')$. In particular, $|X|\geq |C|$, implying in turn that $|C|=3$, a contradiction as $G$ is bipartite. Thus $G$ is a tree. It is immediate from our minimality assumption that every leaf of $G$ belongs to $X$.
\end{proof}

We are now ready to prove the following lemma, the workhorse for the proof of \Cref{main-L7}:

\begin{LE}\label{L7-minor}
Let $\mathcal{C}$ be a clean tangled clutter, where the following statements hold:\begin{enumerate}[(a)]
\item $\mathcal{C}$ has rank $7$, and for each $i\in [7]$, the $i\textsuperscript{th}$ connected component of $G(\mathcal{C})$ has bipartition $\{U_i,V_i\}$.
\item For each $L\in \mathbb{L}_7$, $\bigcup_{i\notin L} U_i \cup \bigcup_{j\in L} V_j$ contains a member of $\mathcal{C}$. 
\item For all $L\in \mathbb{L}_7$ but at most one, $\bigcup_{j\in L} V_j$ is a cover of $\mathcal{C}$.
\end{enumerate}
Then $\mathcal{C}$ has an $\mathbb{L}_7$ minor.
\end{LE}

\noindent (In (b), $\bigcup_{i\notin L} U_i \cup \bigcup_{j\in L} V_j$ for each $L\in \mathbb{L}_7$ must in fact be a member by \Cref{core->cuboid}; but the proof is easier to read given the current version of (b).)

\begin{proof}
Let $G:=G(\mathcal{C})$.

\begin{claim} 
Take a subset $L\subseteq [7]$ such that $|L|\leq 3$ and $\bigcup_{i\in L}V_i$ is a cover. Then $L\in \mathbb{L}_7$. Moreover, $\bigcup_{i\in L}V_i$ contains a minimal cover of cardinality three picking one element from each $V_i,i\in L$.
\end{claim}
\begin{cproof}
As (b) holds, $\bigcup_{i\in L}V_i$ intersects each $\bigcup_{i\notin L} U_i \cup \bigcup_{j\in L} V_j, L\in \mathbb{L}_7$, implying in turn that $L$ is a cover of $\mathbb{L}_7$. As $b(\mathbb{L}_7)=\mathbb{L}_7$ and $|L|\leq 3$, it follows that $L\in \mathbb{L}_7$. The second part follows from~\Cref{mono}.
\end{cproof}

\begin{claim} 
For each $L\in \mathbb{L}_7$, $\bigcup_{i\in L}V_i$ is a cover. 
\end{claim}
\begin{cproof}
We may assume because of (c) that for each $L\in \mathbb{L}_7-\{\{3,5,7\}\}$, $\bigcup_{i\in L}V_i$ contains a minimal cover $B_L$; we may assume by Claim~1 that $B_L$ has cardinality three and picks one element from each $V_i,i\in L$. It remains to prove that $V_3\cup V_5\cup V_7$ is a cover. Suppose otherwise. Let $\mathcal{C}':=\mathcal{C}\setminus (V_5\cup V_7)/ (U_5\cup U_7)$.

Assume in the first case that $\tau(\mathcal{C}')\leq 1$. That is, there is a minimal cover $D\in b(\mathcal{C})$ such that $D\cap (U_5\cup U_7)=\emptyset$ and $|D-(V_5\cup V_7)|\leq 1$. It follows from Claim~1 that $D-(V_5\cup V_7) = \{u\}$ for some $u\in V_3\cup U_1\cup U_2\cup U_3\cup U_4\cup U_6$. Our contrary assumption tells us that $u\notin V_3$. But then $D$ is disjoint from one of $$\bigcup_{i\notin L} U_i \cup \bigcup_{j\in L} V_j,\quad L=\{1,2,3\},\{3,4,6\},$$ a contradiction to (b).

Assume in the remaining case that $\tau(\mathcal{C}')\geq 2$. Then $\mathcal{C}'$ is clean tangled, and $G':=G(\mathcal{C}')$ has $G\setminus (U_5\cup V_5\cup U_7\cup V_7)$ as a subgraph. Then $G'$ is a bipartite graph where for each $i\in \{1,2,3,4,6\}$, $G'[U_i\cup V_i]$ is connected and has bipartition $\{U_i,V_i\}$. Observe that for $L=\{1,4,5\},\{2,4,7\},\{2,5,6\},\{1,6,7\}$, the set $B_L-(V_5\cup V_7)$ is a cardinality-two cover, and therefore a minimum cover, of $\mathcal{C}'$. As a consequence, $G'$ has an edge between $V_1,V_4$, an edge between $V_4,V_2$, an edge between $V_2,V_6$, and an edge between $V_6,V_1$.
Let $U:=U_1\cup U_2\cup V_4\cup V_6$ and $U':=V_1\cup V_2\cup U_4\cup U_6$. Then $G'[U\cup U']$ is connected and has bipartition $\{U,U'\}$. Since $G'[U_3\cup V_3]$ is also connected, $G'$ has at most two connected components. It therefore follows from~\Cref{small-rank}~(i)-(ii) that either $$U\cup U_3,U'\cup V_3\in \mathcal{C}' \quad\text{or}\quad U\cup V_3,U'\cup U_3\in \mathcal{C}'.$$ Observe that $B_L-(V_5\cup V_7) = B_L$ is a cover of $\mathcal{C}'$ for $L=\{1,2,3\},\{3,4,6\}$. However, $B_{\{1,2,3\}} \cap (U\cup U_3)=\emptyset$ and $B_{\{3,4,6\}} \cap (U'\cup U_3)=\emptyset$, a contradiction.

As a result, $V_3\cup V_5\cup V_7$ is a cover, as claimed.
\end{cproof}

We may assume that $\mathcal{C}$ is contraction minimal with respect to being tangled and satisfying (a)-(c). By Claims~1 and~2, for each $L\in \mathbb{L}_7$, there exists a minimal cover $B_L\in b(\mathcal{C})$ of cardinality three picking one element from each $V_i,i\in L$. For each $i\in [7]$, let $$X_i:=V_i\cap \left(\bigcup_{L\in \mathbb{L}_7} B_L\right);$$ notice that $1\leq |X_i|\leq 3$.

\begin{claim} 
For each $i\in [7]$, either $|X_i|=1$ and $|U_i|=|V_i|=1$, or $2\leq |X_i|\leq 3$ and $G[U_i\cup V_i]$ is a tree whose leaves are contained in $X_i$.
\end{claim}
\begin{cproof}
Let $W$ be a subset of $V$ such that (1) $W\subseteq U_i\cup V_i$, (2) $X_i\subseteq W$, (3) $|W|\geq 2$, (4) $G[W]$ is connected, and (5) $W$ is minimal subject to (1)-(4). Let $\{U'_i,V'_i\}$ be the bipartition of $G[W]$ where $U'_i\subseteq U_i$ and $X_i\subseteq V'_i\subseteq V_i$. Notice that if $|X_i|=1$ then $|U'_i|=|V'_i|=1$, and if $2\leq |X_i|\leq 3$ then $G[W]$ must be a tree whose leaves are contained in $X_i$ by~\Cref{trees}. Let $I:=(U_i\cup V_i)-(U'_i\cup V'_i)$. Notice that $\mathcal{C}/ I$ is clean and tangled, and satisfies (a) and (b). Moreover, since $B_L\cap I=\emptyset$ for each $L\in \mathbb{L}_7$, $\mathcal{C}/I$ also satisfies (c). Our minimal choice of $\mathcal{C}$ implies that $I=\emptyset$, so $U'_i=U_i$ and $V'_i=V_i$, thereby finishing the proof of the claim.
\end{cproof}

\begin{claim} 
For each $i\in [7]$, $|X_i|=1$ and $|U_i|=|V_i|=1$.
\end{claim}
\begin{cproof}
Suppose otherwise. We may assume that $G[U_1\cup V_1]$ is not an edge. It then follows from Claim~3 that $2\leq |X_1|\leq 3$ and $G[U_1\cup V_1]$ is a tree whose leaves are contained in $X_1$. Pick a leaf $u$ of the tree $G[U_1\cup V_1]$ that belongs to exactly one of $B_{\{1,2,3\}},B_{\{1,4,5\}},B_{\{1,6,7\}}$, and let $\mathcal{C}':=\mathcal{C}/ u$. Since $u$ is a leaf of $G[U_1\cup V_1]$, $\mathcal{C}'$ is clean and tangled, and satisfies (a) and (b). Moreover, as $u$ belongs to exactly one of $(B_L:L\in \mathbb{L}_7)$, $\mathcal{C}'$ also satisfies (c), a contradiction to the minimality of $\mathcal{C}$.
\end{cproof}

Let $\mathcal{C}':=\mathcal{C}/(U_1\cup \cdots\cup U_7)$. 

\begin{claim} 
$\mathcal{C}' \cong \mathbb{L}_7$.
\end{claim}
\begin{cproof}
We know that $B_L\in b(\mathcal{C}')$ for each $L\in \mathbb{L}_7$, and that by Claim~1, these are the only minimal covers of $\mathcal{C}'$ of cardinality at most three. After a possible relabeling of its elements, we may assume that $\mathcal{C}'$ has ground set $[7]$, and that $B_L=L$ for each $L\in \mathbb{L}_7$. We claim that $b(\mathcal{C}') =\mathbb{L}_7$. Suppose otherwise. Then $b(\mathcal{C}')$ has a member $B$ of cardinality at least four. As $\mathbb{L}_7\subseteq b(\mathcal{C}')$, it follows that $|B|=4$ and $B = [7] - L$ for some $L\in \mathbb{L}_7$. However, $B$ is also a minimal cover of $\mathcal{C}$ that is disjoint from $\bigcup_{i\notin L} U_i \cup \bigcup_{j\in L} V_j$, a contradiction to (b). As a result, $b(\mathcal{C}') =\mathbb{L}_7$, so $\mathcal{C}' = b(\mathbb{L}_7) = \mathbb{L}_7$, as claimed.
\end{cproof}

As a result, $\mathcal{C}'$ has an $\mathbb{L}_7$ minor, thereby finishing the proof of~\Cref{L7-minor}.
\end{proof}

\subsection{Proof of \Cref{main-L7}}\label{sec:fano/main}

Let us start with the following remark about the cocycle space of the Fano matroid:

\begin{RE}\label{PG22}
$\cuboid(\cocycle{PG(2,2)})$ is, after a possible relabeling, a clutter over ground set $\{1,2,\ldots,7,$ $\bar{1},\bar{2},\ldots,\bar{7}\}$ satisfying the following statements: \begin{itemize}
\item the members are $\{\bar{i}:i\in [7]\}$ and $\{i:i\notin L\} \cup \{\bar{j}:j\in L\}$ for all $L\in \mathbb{L}_7$,
\item the cardinality-three minimal covers are $\{\bar{i},\bar{j},\bar{k}\},\{\bar{i},j,k\},\{i,\bar{j},k\},\{i,j,\bar{k}\}$ for all $\{i,j,k\}\in \mathbb{L}_7$,
\item every cardinality-three minimal cover is contained in exactly two members.
\end{itemize}
\end{RE}

Let $S:=\cocycle{PG(2,2)}$. As $S$ is a binary space, it follows from \Cref{transitive} that $S\tr p = S$ for every point $p\in S$. In particular, every member of $\cuboid(S)$ can be treated as the first member $\{\bar{i}:i\in [7]\}$ above.

\begin{PR}\label{main-L7-PR}
Let $\mathcal{C}$ be a clean tangled clutter over ground set $V$ that has a unique fractional packing of value two, and of rank seven. Then $\mathcal{C}$ has an $\mathbb{L}_7$ minor.
\end{PR}
\begin{proof}
Let $G:=G(\mathcal{C})$, and for each $i\in [7]$, let $\{U_i,V_i\}$ be the bipartition of the $i\textsuperscript{th}$ connected component of $G$. Let $y$ be the fractional packing of $\mathcal{C}$ of value two. As $\mathcal{C}$ has rank seven, it follows from \Cref{main-PG} that $\supp(y)$ is a duplication of $\cuboid(\cocycle{PG(2,2)})$. As $\supp(y) \subseteq \core{\mathcal{C}}$, it follows from~\Cref{core->cuboid} and~\Cref{PG22} that, after possibly relabeling and swapping $U_i,V_i, i\in [7]$, the following sets are the members of $\supp(y)$: $$\bigcup_{j=1}^7 V_j \quad\text{and}\quad \bigcup_{i\notin L} U_i \cup \bigcup_{j\in L} V_j \quad \forall L\in \mathbb{L}_7.$$ 

A subset $B\subseteq V$ is a {\it special cover of $\mathcal{C}$} if it is a monochromatic minimal cover intersecting at most three connected components of $G$.

\setcounter{claim_nb}{0}

\begin{claim} 
If $B$ is a special cover of $\mathcal{C}$, then \begin{itemize}
\item there is a unique $L\in \mathbb{L}_7$ such that $B\cap (U_i\cup V_i)\neq \emptyset$ for each $i\in L$, 
\item $\{i\in L: B\cap U_i\neq \emptyset\}$ has even cardinality, and
\item $B$ is contained in exactly two members of $\supp(y)$.
\end{itemize}
\end{claim}
\begin{cproof}
This follows immediately from~\Cref{PG22}.
\end{cproof}

Given a special cover $B$, we refer to $L$ from Claim~1 as the {\it Fano line corresponding to $B$}, and to $\{i\in L : B\cap U_i\neq \emptyset\}$ as the {\it trace of $B$}.

\begin{claim} 
For every Fano line $L\in \mathbb{L}_7$, there are three corresponding special covers with pairwise different traces. 
\end{claim}
\begin{cproof}
Suppose otherwise. We may assume by symmetry between the members of $\mathbb{L}_7$ that $L=\{1,2,3\}$. By Claim~1, every special cover corresponding to $L$ has trace $\emptyset,\{1,2\},\{1,3\}$ or $\{2,3\}$. We may assume by symmetry between the members of $\cuboid(\cocycle{PG(2,2)})$ that every special cover corresponding to line $L$, if any, has trace $\{1,2\}$ or $\{1,3\}$. Let $\mathcal{C}':=\mathcal{C}\setminus V_1/U_1$ and $G':=G(\mathcal{C}')$. By \Cref{fp-recursive-2}, $\mathcal{C}'$ is clean and tangled and has a unique fractional packing of value two, and given that $z$ is the fractional packing of $\mathcal{C}'$ of value two, $\supp(z) =\supp(y)\setminus V_1/U_1$. In particular, $\supp(z)$ is a duplication of $\cuboid(\cocycle{PG(1,2)})$ by \Cref{PG->PG}.
\Cref{main-PG} applied to $\mathcal{C}'$ now tells us that $$\rank{\mathcal{C}'} = 2^2-1=3.$$
Observe that for $i\in [7]-\{1\}$, $G[U_i\cup V_i]\subseteq G'[U_i\cup V_i]$, so $G'[U_i\cup V_i]$ is connected. Let us refer to the edges of $G'$ not contained in any $G'[U_i\cup V_i],i\in [7]-\{1\}$ as {\it crossing} edges. We claim that \begin{quote}
$(\star)$ for every crossing edge $\{u,v\}$, either $\{u,v\}\subseteq U_4\cup U_5$, $\{u,v\}\subseteq V_4\cup V_5$, $\{u,v\}\subseteq U_6\cup U_7$ or $\{u,v\}\subseteq V_6\cup V_7$.\end{quote} To this end, pick distinct $i,j\in [7]-\{1\}$ such that $u\in U_i\cup V_i$ and $v\in U_j\cup V_j$. Then $V_1\cup \{u,v\}$ contains a minimal cover of $\mathcal{C}$, which is inevitably special. It therefore follows from Claim~1 that $\{1,i,j\}\in \mathbb{L}_7$, and either $\{u,v\}\subseteq U_i\cup U_j$ or $\{u,v\}\subseteq V_i\cup V_j$. Since there is no special cover corresponding to line $\{1,2,3\}$ and trace either $\emptyset,\{2,3\}$, it follows that $\{i,j\}=\{4,5\}$ or $\{6,7\}$, so $(\star)$ holds. However, $(\star)$ implies that $G'$ has at least four connected components, so $\rank{\mathcal{C}'}\geq 4$, a contradiction.
\end{cproof}

\begin{claim} 
There is a member of $\supp(y)$ that contains six special covers corresponding to different Fano lines.
\end{claim}
\begin{cproof}
By Claim~2, there are $21=7\times 3$ special covers $B_1,\ldots,B_{21}$ such that for distinct $i,j\in [21]$, if $B_i$ and $B_j$ correspond to the same Fano line, then they have different traces. By Claim 1, each $B_i, i\in [21]$ is contained in exactly two members of $\supp(y)$. As a result, there is a member of $\supp(y)$ containing at least $\frac{21\times 2}{8}>5$ special covers among $B_1,\ldots,B_{21}$, as required.
\end{cproof}

We may assume that $\bigcup_{j=1}^7 V_j$ contains six special covers corresponding to different Fano lines. As $\mathcal{C}$ satisfies conditions (a)-(c), we may apply \Cref{L7-minor} to conclude that $\mathcal{C}$ has an $\mathbb{L}_7$ minor, as required.
\end{proof}

We are now ready for the main result of this section:

\begin{proof}[Proof of \Cref{main-L7}]
Let $\mathcal{C}$ be a clean tangled clutter with a unique fractional packing of value two and of rank more than three. Let $y$ be the fractional packing of $\mathcal{C}$ of value two. It then follows from \Cref{main-PG} that for some integer $k\geq 3$, $\mathcal{C}$ has rank $2^k-1$, and $\supp(y)$ is a duplication of $\cuboid(\cocycle{PG(k-1,2)})$. We prove by induction on $k\geq 3$ that $\mathcal{C}$ has an $\mathbb{L}_7$ minor. The base case $k=3$ follows from \Cref{main-L7-PR}. For the induction step, assume that $k\geq 4$. Let $\{U,U'\}$ be a connected component of $G(\mathcal{C})$, and let $\mathcal{C}':=\mathcal{C}\setminus U/U'$. By \Cref{fp-recursive-2}, $\mathcal{C}'$ is clean tangled and has a unique fractional packing of value two, and if $z$ is the fractional packing of $\mathcal{C}'$ of value two, then $\supp(z) = \supp(y)\setminus U/U'$. In particular, $\supp(z)$ is a duplication of $\cuboid(\cocycle{PG(k-2,2)})$ by \Cref{PG->PG}. Thus $\mathcal{C}'$ has rank $2^{k-1}-1$ by \Cref{main-PG}, so by the induction hypothesis, $\mathcal{C}'$ and therefore $\mathcal{C}$ has an $\mathbb{L}_7$ minor, thereby completing the induction step. This finishes the proof of \Cref{main-L7}.
\end{proof}

\subsection{Ideal clutters and an application}\label{sec:fano/apps}

\Cref{main-L7} has a geometric consequence; let us elaborate. A clutter $\mathcal{C}$ over ground set $V$ is {\it ideal} if the associated set covering polyhedron $$\left\{x\in \mathbb{R}^V_+:\sum_{v\in C}x_v\geq 1\quad C\in \mathcal{C}\right\}$$ is integral~\cite{Cornuejols94} (see also~\cite{Abdi-thesis}). It can be readily checked by the reader that the deltas and $\mathbb{L}_7$ are non-ideal clutters. (In fact, every identically self-blocking clutter different from $\{\{a\}\}$ is non-ideal~\cite{Abdi-ISBC}.) It can also be readily checked that every extended odd hole is non-ideal. It is well-known that a clutter is ideal if and only if its blocker is ideal~\cite{Fulkerson70,Lehman79}. In particular, the blocker of an extended odd hole is also non-ideal. Moreover, if a clutter is ideal, so is every minor of it~\cite{Seymour77}. Thus, every ideal clutter is clean.

\begin{theorem}\label{main-geometric}
Let $\mathcal{C}$ be a clean tangled clutter. If $\conv(\setcore{\mathcal{C}})$ is a simplex, then at least one of the following statements holds: \begin{enumerate}[(i)]
\item $\setcore{\mathcal{C}} = \{0,1\}$, i.e. $\core{\mathcal{C}}$ consists of two members that partition the ground set,
\item $\setcore{\mathcal{C}} \cong \{000,110,101,011\}$, i.e. $\core{\mathcal{C}}$ is a duplication of $Q_6$, or
\item $\mathcal{C}$ is non-ideal.
\end{enumerate}
\end{theorem}
\begin{proof}
Let $r:=\rank{\mathcal{C}}$. If $r>3$, then $\mathcal{C}$ has an $\mathbb{L}_7$ minor by \Cref{main-L7}, so (iii) holds in particular. Otherwise, $1\leq r\leq 3$. It follows from \Cref{main-PG} that $\setcore{\mathcal{C}}$ is isomorphic to either $\cocycle{PG(0,2)}=\{0,1\}$ or $\cocycle{PG(1,2)}=\{000,110,101,011\}$, so either (i) or (ii) holds, as required.
\end{proof}

Observe that the statement of Theorem~\ref{main-geometric} is geometric while our proof is purely combinatorial, further stressing the synergy between the combinatorics and the geometry of clean tangled clutters. Recently, the authors gave an example of an infinite class of clean tangled clutters (more precisely, \emph{ideal minimally non-packing} clutters with covering number two) that belong to category (ii) of \Cref{main-geometric}~\cite{Abdi-infinite}.

\section{Future directions for research}\label{sec:conclusion}

Clean tangled clutters were the subject of study in this paper. It was proved that the convex hull of the setcore of every such clutter is a full-dimensional polytope containing the center point of the hypercube in its interior (\Cref{setcore}). The setcore has a simplicial convex hull if and only if it is the cocycle space of a projective geometry over the two-element field (\Cref{main-PG}). Moreover, if the setcore has a simplicial convex hull, then the clutter has rank at most three or it has an $\mathbb{L}_7$ minor (\Cref{main-L7}).

We conclude the paper with three directions for future research.

Our results expose a fruitful interplay between the combinatorics and the geometry of clean tangled clutters. Further along these lines, and an extension of \Cref{setcore}, is a min-max relation that holds for such clutters and relates a geometric parameter to a combinatorial parameter. Given a clean tangled clutter $\mathcal{C}$ of rank $r$, denote by $\mu(\mathcal{C})$ the minimum cardinality of a monochromatic cover of $\mathcal{C}$, and by $\depth{\mathcal{C}}$ the maximum $d$ such that there exists a $d$-dimensional subhypercube of $\{0,1\}^{r}$ disjoint from $\setcore{\mathcal{C}}$. (If there is no monochromatic cover then $\mu(\mathcal{C}):=\infty$, and if $\setcore{\mathcal{C}}=\{0,1\}^{r}$ then $\depth{\mathcal{C}}:=-\infty$.)

\begin{theorem}[\cite{Abdi-duality}]
Let $\mathcal{C}$ be a clean tangled clutter. Then $\rank{\mathcal{C}}-\mu(\mathcal{C})=\depth{\mathcal{C}}$.
\end{theorem}

A clutter $\mathcal{C}$ \emph{embeds} $PG(k-2, 2)$ if some subset of $\mathcal{C}$ is a duplication of the cuboid of $\cocycle{PG(k-2, 2)}$. This notion was defined in~\cite{Abdi-kwise}. We conjecture that,

\begin{CN}\label{PG-con}
Every clean tangled clutter embeds a projective geometry over the two-element field.
\end{CN}

\noindent This conjecture has an intimate connection to \emph{dyadic} fractional packings of value two in clean tangled clutters; see~\cite{Abdi-dyadic}. Observe that \Cref{main-PG} proves \Cref{PG-con} when the setcore of the clutter has a simplicial convex hull. 

The following variant of \Cref{PG-con} has also been conjectured:

\begin{CN}[\cite{Abdi-kwise}]\label{ideal-con}
There exists an integer $\ell\geq 3$ such that every ideal tangled clutter embeds one of $PG(0,2),\ldots,PG(\ell-1,2)$.
\end{CN}

\noindent This conjecture has an intimate connection to the idealness of \emph{$k$-wise intersecting clutters}~\cite{Abdi-kwise,Abdi-kwise-IPCO}. Observe that \Cref{main-geometric}, which is a consequence of \Cref{main-L7}, proves \Cref{ideal-con} for $\ell=3$ when the setcore of the clutter has a simplicial convex hull (in fact, $\ell=2$ suffices here).

\section*{Acknowledgements}

We would like to thank Tony Huynh, Bertrand Guenin, Dabeen Lee, and Levent Tun\c{c}el for fruitful discussions about various parts of this work. This work was supported by ONR grant 00014-18-12129, NSF grant CMMI-1560828, and NSERC PDF grant 516584-2018.

{\small \bibliographystyle{abbrv}\bibliography{references}}

\end{document}